\documentclass{amsart}
\usepackage{amsmath,amssymb,amsthm,mathrsfs}

\newtheorem{theorem}{Theorem}[section]
\newtheorem{lemma}[theorem]{Lemma}
\newtheorem{corollary}[theorem]{Corollary}
\newtheorem{proposition}[theorem]{Proposition}

\newtheorem{remark}[theorem]{Remark}
\newtheorem{definition}[theorem]{Definition}

\numberwithin{equation}{section}

\newcommand{\BC}{{\mathbb C}}\newcommand{\BD}{{\mathbb D}}

\newcommand{\BZ}{{\mathbb Z}}

\newcommand{\cD}{{\mathcal D}}

\newcommand{\cL}{{\mathcal L}}
\newcommand{\cN}{{\mathcal N}}

\newcommand{\cS}{{\mathcal S}}
\newcommand{\cU}{{\mathcal U}}
\newcommand{\cX}{{\mathcal X}}
\newcommand{\cY}{{\mathcal Y}}
\newcommand{\bA}{{\mathbf A}}\newcommand{\bB}{{\mathbf B}}
\newcommand{\bC}{{\mathbf C}}\newcommand{\bD}{{\mathbf D}}

\newcommand{\bW}{{\mathbf W}}

\newcommand{\fH}{{\mathfrak H}}

\newcommand{\fL}{{\mathfrak L}}

\newcommand{\ga}{\gamma}\newcommand{\Ga}{\Gamma}
\newcommand{\de}{\delta}

\newcommand{\Si}{\Sigma}

\newcommand{\im}{\textup{Im\,}}

\newcommand{\kr}{\textup{Ker\,}}

\newcommand{\spec}{r_\textup{spec}}

\newcommand{\mat}[2]{\ensuremath{\left[\begin{array}{#1}#2\end{array}\right]}}

\newcommand{\sbm}[1]{\left[\begin{smallmatrix}#1\end{smallmatrix}\right]}

\newcommand{\tu}[1]{\textup{#1}}

\newcommand{\half}{\frac{1}{2}}
\newcommand{\ands}{\quad\mbox{and}\quad}

\newcommand{\wtil}{\widetilde}

\newcommand{\Obs}{\operatorname{Obs}\,}
\newcommand{\Rea}{\operatorname{Rea}\,}

\newcommand{\bu}{{\mathbf u}}

\newcommand{\bx}{{\mathbf x}}
\newcommand{\by}{{\mathbf y}}

\newcommand{\frakL}{{\mathfrak L}}

\begin{document}

\title[Infinite Dimensional Bounded Real Lemma I]{Standard versus Strict Bounded Real Lemma
with infinite-dimensional state space I: \\ The State-Space-Similarity Approach}

\author{J.A. Ball}
\address{J.A. Ball, Department of Mathematics, Virginia Tech,
Blacksburg, VA 24061-0123, USA}
\email{joball@math.vt.edu}

\author{G.J. Groenewald}
\address{G.J. Groenewald, Department of Mathematics, Unit for BMI,
North-West University, Potchefstroom 2531, South Africa}
\email{Gilbert.Groenewald@nwu.ac.za}

\author{S. ter Horst}
\address{S. ter Horst, Department of Mathematics, Unit for BMI,
North-West University, Potchefstroom 2531, South Africa}
\email{Sanne.TerHorst@nwu.ac.za}

\thanks{This work is based on the research supported in part by the National
Research Foundation of South Africa (Grant Numbers 93039, 90670, and 93406).}

\begin{abstract}
The Bounded Real Lemma, i.e., the state-space linear matrix inequality characterization
(referred to as Kalman-Yakubovich-Popov or KYP inequality) of when an input/state/output
linear system satisfies a dissipation inequality, has recently been studied for infinite-dimensional
discrete-time systems in a number of different settings: with or without stability assumptions,
with or without controllability/observability assumptions, with or without strict inequalities.
In these various settings, sometimes unbounded solutions of the KYP inequality are required
while in other instances bounded solutions suffice.
In a series of reports we show how these diverse results can be reconciled and unified.
This first instalment focusses on the state-space-similarity approach to the bounded real lemma.
We shall show how these results can be seen as corollaries of a new State-Space-Similarity theorem
for infinite-dimensional linear systems.
\end{abstract}



\maketitle


\section{Introduction}\label{S:intro}
We consider the discrete-time linear system
\begin{equation}\label{dtsystem}
\Si:=\left\{
\begin{array}{ccc}
\bx(n+1)&=&A \bx(n)+B \bu(n),\\
\by(n)&=&C \bx(n)+D \bu(n),
\end{array}
\right. \qquad (n\in\BZ)
\end{equation}
where $A:\cX\to\cX$, $B:\cU\to\cX$, $C:\cX\to\cY$ and $D:\cU\to\cY$
are bounded linear Hilbert space operators, i.e., $\cX$, $\cU$ and
$\cY$ are Hilbert spaces and the {\em system matrix} associated with
$\Si$ takes the form
\begin{equation}\label{sysmat}
M=\mat{cc}{A&B\\ C& D}:\mat{cc}{\cX\\ \cU}\to\mat{c}{\cX\\ \cY}.
\end{equation}
We refer to the pair $(C,A)$ as an {\em output pair} and to the pair
$(A,B)$ as an {\em input pair}. In this case input sequences
$\bu=(\bu(n))_{n\in\BZ}$, with $\bu(n)\in\cU$, are mapped to output
sequences $\by=(\by(n))_{n\in\BZ}$, with $\by(n)\in\cY$, through the
state sequence $\bx=(\bx(n))_{n\in\BZ}$, with $\bx(n)\in \cX$. With
the system $\Si$ we associate the {\em transfer function} given by
\begin{equation}\label{trans}
F_\Si(z)=D+zC(I-zA)^{-1}B.
\end{equation}
Since $A$ is bounded, $F_\Si$ is defined and analytic on a
neighborhood of $0$ in $\BC$. We shall be interested in the case
where $F_\Si$ admits an analytic continuation to the open unit disk
$\BD$ such that the supremum norm $\|F_\Si\|_\infty$ of $F_\Si$ over
$\BD$ is at most one, i.e., $F_\Si$ has analytic continuation to a
function in the Schur class
\[
  \cS(\cU, \cY) = \left\{ F \colon {\mathbb D}
  \underset{\text{holo}}\mapsto \cL(\cU, \cY) \colon \| F(z) \| \le 1
  \text{ for all } z \in {\mathbb D}\right\}.
\]

A well-known sufficient condition for this to be the case is that the
system matrix $M$ be contractive. We review the elementary argument.
Note first that $\left\| \sbm{ A & B \\ C & D } \right\| \le 1$
implies that $\|A \| \le 1$ and hence $|z| < 1$ implies that $\| z A
\| < 1$. Therefore $I - zA$ is boundedly invertible, and hence the
transfer function $F_{\Sigma}$ is well-defined and analytic on the
open unit disk ${\mathbb D}$.  For any $u \in \cU$ and $z \in
{\mathbb D}$ we have the identity
\begin{equation}\label{syseq1}
\begin{bmatrix} A & B \\ C & D \end{bmatrix}  \begin{bmatrix} z(I-
zA)^{-1} B u \\ u \end{bmatrix}
 = \begin{bmatrix}  (I - zA)^{-1} B u \\ F_{\Sigma}(z) u
\end{bmatrix}.
\end{equation}
For simplicity let us set $x = z(I - zA)^{-1} B u$ and $x' = (I -
zA)^{-1} B u$, so that we can rewrite \eqref{syseq1} as the feedback
system
\begin{equation}\label{syseq2}
\begin{aligned}
\begin{bmatrix} A & B \\ C & D \end{bmatrix}
\begin{bmatrix} x \\ u \end{bmatrix} &  =  \begin{bmatrix} x' \\
    F_{\Si}(z) u \end{bmatrix}, \\  x & =  z x'.
\end{aligned}
\end{equation}
The fact that $\| \sbm{ A & B \\ C & D} \| \le 1$ now implies that
 \[
 \|  x' \|^{2} + \| F_{\Sigma}(z) u \|^{2} \le \| x \|^{2} + \| u
\|^{2}.
\]
Rewrite this and use that $\|x\|^{2} = |z|^{2} \| x'\|^{2} \le  \|
x'\|^{2}$ to get
\[
  \| F_{\Sigma}(z) u \|^{2} \le \| x \|^{2} - \| x'\|^{2} + \|u
  \|^{2} \le \| u \|^{2}.
\]
Since $u \in \cU$ and $z \in {\mathbb D}$ were chosen arbitrarily, we
can conclude that  $\|F_{\Sigma}(z) \| \le 1$ for all $z\in\BD$,
i.e., $F_\Si$ is in the Schur class $\cS(\cU, \cY)$.
For a circuit-theoretic perspective on this argument, we refer to the paper of Helton-Zemanian
\cite{HZ}.

The same argument goes through if we suppose that the system matrix
$M$ is contractive when some other equivalent norm
\begin{equation}   \label{Hnorm}
\|x\|_H = \langle Hx, x \rangle^{\half}, \mbox{ with $H$ strictly positive-definite on $\cX$},
\end{equation}
is used on the state space.
Here we use the conventions: given a selfadjoint operator $H$ on a Hilbert space $\cX$, we say:
\begin{enumerate}
\item[(1)] $H$ is {\em strictly positive-definite} (written $H \succ 0$) if there is a $\delta > 0$ so that
$\langle H x , x \rangle \ge \delta \| x \|^2$ for all $x \in \cX$,
\item[(2)] $H$ is {\em positive-definite} if $\langle H x, x \rangle > 0$ for all nonzero $x \in \cX$, and
\item[(3)] $H$ is {\em positive-semidefinite} if $\langle H x, x \rangle \ge 0$ for all $x \in \cX$.
\end{enumerate}
Note that there is no distinction between {\em strictly positive-definite} and {\em positive-defi\-nite} if $\cX$ is finite-dimensional.
The condition that the system matrix $M = \sbm{ A & B \\ C & D}$ is contractive with the $H$-norm \eqref{Hnorm} used on the state space
translates to:  {\em there exists a bounded strictly positive-definite operator $H$ on $\cX$ so that the {\em
Kalman-Yakubovich-Popov inequality} holds:}
\begin{equation}  \label{KYP1}
\begin{bmatrix} A & B \\ C & D \end{bmatrix}^{*}
\begin{bmatrix} H &  0 \\ 0 & I_{\cY} \end{bmatrix}
\begin{bmatrix} A & B \\ C & D\end{bmatrix}
\preceq \begin{bmatrix} H & 0 \\ 0 & I_{\cU}\end{bmatrix}.
\end{equation}
Indeed, in this case $F_\Si$ is also the transfer function of the
system $\wtil{\Si}$ with contractive system matrix $\wtil{M}$
obtained after a state space similarity with $H^{\half}$, i.e.,
\[
\wtil{M}=\mat{cc}{H^{\half} A H^{-\half} & H^{\half} B\\ C H^{-\half}
&
D}.
\]
Since $A$ is similar to the contraction $H^{\half} A H^{-\half}$, we
have $\spec(A)\leq 1$, so that $F_\Si$ in \eqref{trans} is defined
and analytic on $\BD$. Thus the  KYP inequality \eqref{KYP1} can
have a bounded strictly positive-definite solution $H$ only in case $\spec(A)\leq
1$.

For future reference, note that the KYP inequality \eqref{KYP1} can
be rewritten in spatial form as
\begin{equation} \label{KYP1a}
\left\| \begin{bmatrix} H^{\half} & 0 \\ 0 & I_{\cU}\end{bmatrix}
\begin{bmatrix} x \\ u \end{bmatrix} \right\|^{2}
  - \left\| \begin{bmatrix} H^{\half} & 0 \\ 0 & I_{\cY}\end{bmatrix}
\begin{bmatrix} A & B \\ C & D \end{bmatrix}
\begin{bmatrix} x \\ u \end{bmatrix} \right\|^{2} \ge 0\qquad ( x \in
\cX,\, u \in \cU).
\end{equation}

By a Schur-complement
argument (with the obvious invertibility assumption), the KYP
inequality \eqref{KYP1} can be converted into Riccati form
\[
H - A^{*}HA - C^{*}C - (A^{*}HB + C^{*}D)(I - B^{*}H B -
D^{*}D)^{-1}(B^{*}H A + D^{*}C) \succeq 0
\]
which one can then attempt to solve for $H$ directly.  We do not
pursue this direction and will refer to \eqref{KYP1} as the KYP
inequality for the unknown $H$. See \cite{AKP16} for a recent
treatment of this Riccati
form of the KYP inequality, which also considers the case where
equality occurs.

The Bounded Real Lemma is concerned with the converse question:
{\em  Given a system $\Si$ as in \eqref{dtsystem} with system matrix
$M = \sbm{A & B \\ C & D}$ and transfer function $F_\Si(z) = D + z C
(I - zA)^{-1} B$, defined at least in a neighborhood of $0$, give
explicit conditions in terms of
$M = \sbm{ A & B \\ C & D }$ under which $F_\Si$ has analytic
continuation to a function in the Schur class $\cS(\cU, \cY)$?} We
mention two such versions for the finite-dimensional situation.

\begin{theorem}[Standard Bounded  Real Lemma (see \cite{AV})]
\label{T:BRLfinstan}
Suppose that $\Si$ is a discrete-time linear system as in
\eqref{dtsystem} with $\cX$, $\cU$ and $\cY$ finite dimensional, say
$\cU = {\mathbb C}^{r}$, $\cY = {\mathbb C}^{s}$, $\cX = {\mathbb
C}^{n}$, so that the system matrix $M$ has the form
\begin{equation}\label{findimsys}
M = \begin{bmatrix} A & B \\ C & D \end{bmatrix} \colon
 \begin{bmatrix} {\mathbb C}^{n} \\ {\mathbb C}^{r} \end{bmatrix} \to
     \begin{bmatrix} {\mathbb C}^{n} \\ {\mathbb C}^{s} \end{bmatrix}
\end{equation}
and the transfer function  $F_{\Sigma}$ is equal to a
rational matrix function of size $s \times r$. Assume that the
realization $(A,B,C,D)$ is {\em minimal}, i.e., the output pair
$(C,A)$ is {\em observable} and the input pair $(A,B)$ is {\em
controllable}:
\begin{equation}\label{obscontr}
 \bigcap_{k=0}^{n} \kr C A^{k} = \{0\}\ands
\textup{span}_{k=0,1,\dots, n-1} \im A^{k} B  = \cX = {\mathbb C}^{n}.
\end{equation}
Then $F_{\Sigma}$ is in the Schur class $\cS({\mathbb C}^{r},
{\mathbb C}^{s})$
if and only if
there is a $n \times n$ positive-definite matrix $H$ satisfying the
KYP inequality \eqref{KYP1}.
\end{theorem}

In the strict version of the Bounded Real Lemma, one replaces the
minimality condition with a stability condition to characterize the
{\em strict Schur class} $\cS^{o}(\cU, \cY)$:
\[
  \cS^{o}(\cU, \cY) =\left \{ F \colon {\mathbb D}
\underset{\text{holo}}
  \mapsto \cL(\cU, \cY) \colon \sup_{z \in {\mathbb D}} \| F(z) \|
\le \rho \text{ for some }
  \rho < 1\right\}.
\]
Then we have the following result.

\begin{theorem}[Strict Bounded  Real Lemma (see
\cite{PAJ})]\label{T:BRLfinstrict}
Suppose that the dis\-crete-time linear system $\Si$  is as in
\eqref{dtsystem} with  $\cX$, $\cU$ and $\cY$ finite dimensional, say
$\cU = {\mathbb C}^{r}$, $\cY = {\mathbb C}^{s}$, $\cX = {\mathbb
C}^{n}$, i.e., the system matrix $M$ is as in \eqref{findimsys}.
Assume that $A$ is {\em stable}, i.e., all eigenvalues of $A$ are
inside the open unit disk $\BD$, so that $\spec(A) < 1$ and the
transfer function $F_{\Si}(z)$ is analytic on a neighborhood of
$\overline{\BD}$.  Then
$F_{\Si}(z)$ is in the strict Schur class $\cS^{o}({\mathbb C}^{r},
{\mathbb C}^{s})$ if and only if there is a positive-definite matrix
$H \in {\mathbb C}^{n \times n}$ so that the strict KYP inequality
holds:
\begin{equation} \label{KYP2}
\begin{bmatrix} A & B \\ C & D \end{bmatrix}^{*} \begin{bmatrix} H &
    0 \\ 0 & I_{\cY} \end{bmatrix} \begin{bmatrix} A & B \\ C & D
\end{bmatrix} \prec \begin{bmatrix} H & 0 \\ 0 & I_{\cU}
\end{bmatrix}.
\end{equation}
\end{theorem}

The discussion above concerning the sufficiency of the existence of a solution $H$ to the KYP-inequality for
$\Sigma$ for $S_\Sigma$ to be in the Schur class suggests the following proof of the necessity side of
Theorem \ref{T:BRLfinstan}  based on the Kalman State-Space-Similarity Theorem from Linear Systems Theory.
Suppose that $\Sigma$ is a finite-dimensional minimal system with system matrix $M = \sbm{A & B \\ C & D}$
such that $F_\Sigma$ has analytic continuation to a Schur-class function.   It is known from circuit theory (see e.g.\ \cite{AV}) that the rational Schur-class function $F_\Sigma$ also has a realization as $F_{\Sigma'}$
 where $\Sigma'$ is a system with contractive system matrix
$M' = \sbm{A' & B' \\ C' & D}$.  By using Kalman reduction theory, we may suppose that $\Sigma'$ is
controllable and observable (i.e., minimal), and hence that $\Sigma$ and $\Sigma'$ are both minimal.  Then the
Kalman State-Space-Similarity theorem implies that there is a bounded invertible matrix $\Gamma$ so that
$M' = \sbm{ \Gamma A \Gamma^{-1} & \Gamma B \\ C \Gamma^{-1} & D}$.  Since $\| M' \| \le 1$, it is easy
to check that $H = \Gamma^* \Gamma$ is a positive-definite solution of the KYP-inequality for the
system $\Sigma$.  As for Theorem \ref{T:BRLfinstrict}, the proof of Petersen-Anderson-Jonkheere uses a
regularization technique to reduce the strict Bounded Real Lemma to the standard Bounded Real Lemma.

For the case where the state space $\cX$ and the input/output spaces
$\cU$ and $\cY$ are all allowed to be infinite-dimensional, the
results on the Bounded Real Lemma are more recent.  It turns out that
the generalizations of
Theorems \ref{T:BRLfinstan}  and \ref{T:BRLfinstrict} to the
infinite-dimensional situation are quite different, in that the first
involves unbounded operators while the second does not.

For an infinite-dimensional system $\Si$ as in \eqref{dtsystem} much
depends on what is meant by controllable and observable.  Here are a
few possibilities.

\begin{definition}\label{Dconobs}
Let $(C,A)$ be an output pair and $(A,B)$ an input pair. Define the
{\em reachability space} $\Rea (A|B)$ and the {\em observability
space} $\Obs (C|A)$ by
\begin{align}
\Rea (A|B)&=\tu{span}\{\im A^k B\colon k=0,1,2,\ldots\},
\label{ReachSpace}\\
\Obs (C|A)&=\tu{span}\{\im A^{*k} C^*\colon k=0,1,2,\ldots\}=\Rea
(A^*|C^*), \label{ObsSpace}
\end{align}
or, in the terminology of Opmeer-Staffans \cite[Definition
3.1]{OpmeerStaffans2008}, $\Rea(A|B)$ is the set of {\em finite-time
reachable states for the input pair $(A,B)$}, while $\Obs (C|A)$ is
the set of {\em finite-time reachable states for the input pair
$(A^{*},C^{*})$}.
We say that the pair $(C,A)$ is {\em exactly observable} if
$\Obs (C|A)=\cX$ and {\em approximately observable} (or simply {\em
observable} for short) if $\Obs (C|A)$ is dense in $\cX$. Note that
$(C,A)$ being observable is
equivalent to $\bigcap_{n=0}^{\infty} \kr C A^{n} = \{0\}$. Similarly,
we say that the pair $(A,B)$ is {\em exactly controllable} if $\Rea
(A|B)=\cX$ and {\em approximately controllable} (or simply {\em
controllable} for short) if $\Rea (A|B)$ is dense in $\cX$.

Another notion of minimality involves the observability operator
$\bW_o$ and
controllability operator $\bW_c$ associated with the system $\Si$,
which
in the present context may be unbounded operators; see
\eqref{bWo1}-\eqref{bWc*2}
for their definitions and Propositions \ref{P:WcWo'} and  \ref{P:WcWo} for some of
their properties.
We then say that $\Si$ is {\em $\ell^2$-exactly observable} in case $\bW_o$ is densely defined
and has adjoint operator $\bW_o^*$ (which is automatically closed and densely defined)
which is surjective ($\cX=\bW_o^* \cD(\bW_o^*)$).
We say that $\Si$ is {\em $\ell^2$-exactly controllable}
in case the adjoint controllability operator $\bW_c^*$ is densely defined and has adjoint
operator, the controllability operator $\bW_c$, (also automatically closed and densely defined)
which is surjective ($\cX=\bW_c \cD(\bW_c)$).
If $\Sigma$ is a system with system matrix $M = \sbm{ A & B \\ C &
D}$ as in \eqref{sysmat}, we say that $\Sigma$ is
{\em controllable/exactly controllable/$\ell^{2}$-exactly
controllable} if
the input pair $(A,B)$ is controllable/exactly
controllable/$\ell^{2}$-exactly controllable
respectively.  Similarly, we say that $\Sigma$ is {\em
observable/exactly
observable/$\ell^{2}$-exactly observable} if the output pair $(C, A)$
is  observable/exactly  observable/$\ell^{2}$-exactly observable
respectively.
In case $\Sigma$ is both controllable/exactly
controllable/$\ell^{2}$-exactly controllable and observable/exactly
observable/$\ell^{2}$-exactly observable, we say that
the system $\Si$ {\em minimal/exactly  minimal/$\ell^{2}$-exactly
minimal}, respectively.
\end{definition}

As we shall see, either notion of exact controllability/observability
implies (approximate) controllability/observability, but in general
neither notion of exact controllability/observability implies the
other, except with some additional hypotheses imposed (see
Proposition \ref{P:control-observe} below).
Using these notions we obtain the following variation on Theorem
\ref{T:BRLfinstan}.

\begin{theorem}[Infinite-dimensional standard Bounded Real Lemma]
\label{T:BRLinfstan}
Let $\Si$ be a discrete-time linear system as in
\eqref{dtsystem} with system matrix $M$ as in \eqref{sysmat} and
transfer function $F_\Si$ defined by \eqref{trans}.
\begin{enumerate}
\item[(1)] Suppose that the system $\Si$ is minimal, i.e., the input pair
$(A,B)$ is controllable and the output pair $(C,A)$ is observable.
Then the transfer function $F_{\Sigma}$ has an analytic continuation
to a function in the Schur class $\cS(\cU, \cY)$ if and only if there
exists a generalized positive-semidefinite solution $H$ of the
KYP-inequality \eqref{KYP1} in the following generalized sense:  $H$
is a closed, possibly unbounded, densely defined, injective, positive-definite operator on $\cX$ with domain $\cD(H^{\half})$ satisfying
\begin{equation} \label{KYP1b'}
A \cD(H^{\half}) \subset \cD(H^{\half}), \quad B \cU \subset
\cD(H^{\half}),
\end{equation}
and the spatial form of the KYP-inequality holds on the appropriate
domain:
\begin{equation}\label{KYP1b}
\left\| \begin{bmatrix} H^{\half}\!&\! 0 \\ 0 \!&\! I_{\cU}
\end{bmatrix} \begin{bmatrix} x \\ u \end{bmatrix} \right\|^{2}
- \left\| \begin{bmatrix} H^{\half} \!&\! 0 \\ 0 \!&\! I_{\cY}
\end{bmatrix} \begin{bmatrix} A \!&\! B \\ C \!&\! D \end{bmatrix}
\begin{bmatrix} x \\ u \end{bmatrix} \right\|^{2} \ge 0 \ \
(x \in \cD(H^{\half}),\, u \in \cU).
\end{equation}

\item[(2)] Suppose that $\Sigma$ is exactly minimal.  Then the transfer function
$F_{\Sigma}$ has an
analytic continuation to a function in the Schur class $\cS(\cU,
\cY)$ if and only if there exists a bounded strictly positive-definite
solution $H$ of the KYP-inequality \eqref{KYP1}. In this case $A$ has a spectral radius of at most one, and hence $F_{\Sigma}$ is in fact analytic on $\BD$.

\item[(3)]  Statement (2) above continues to hold if the ``exactly minimal'' hypothesis is replaced
by the hypothesis that $\Sigma$ be ``$\ell^2$-exactly minimal''.

\end{enumerate}
\end{theorem}

Let us remark here that most texts do not mention statements (2) or (3)  of
Theorem \ref{T:BRLinfstan}; indeed,  arguably it is rare
that a nonrational matrix function has a realization $\Sigma \sim M = \sbm{ A & B \\ C & D}$
which is exactly minimal or $\ell^2$-exactly minimal.
Nevertheless, we
identify a class of examples where $(A,B)$ is in fact both exactly controllable and
$\ell^2$-exactly controllable and
$(C,A)$ is both exactly observable and $\ell^2$-exactly observable;  either of these classes serves as the
stepping stone for our proof of Theorem \ref{T:BRLinfstrict} below which is a simple adaptation of the
regularization technique of Petersen-Anderson-Jonkheere \cite{PAJ} used in their proof of
the finite-dimensional, rational case.

On the other hand statement (1) of Theorem \ref{T:BRLinfstan} has
appeared in the work of
Arov-Kaashoek-Pik \cite{AKP06} (see Theorems 4.1 and 1.2 there).
Parallel results for the continuous-time setting are developed in the
paper of Arov-Staffans \cite{AS}.

We shall see that all three flavors of the standard Bounded Real Lemma as stated in
Theorem \ref{T:BRLinfstan} follow the sketch outlined above for the finite-dimensional, rational case,
where one uses the contractive realization theorem for (not necessarily rational) Schur-class functions
(see e.g.\ \cite[Theorem 5.2]{BC} or \cite[Theorem VI.3.1]{NF} as well as \cite{AKP97}), and an appropriate
infinite-dimensional State-Space-Similarity theorem as encoded in the following.

\begin{theorem}[Infinite-dimensional State-Space-Similarity theorem]  \label{T:SSS}
Let $\Sigma$ \\ and $\Sigma'$ be two systems with respective system
matrices
$$
M = \begin{bmatrix} A & B \\ C & D \end{bmatrix} \colon \begin{bmatrix} \cX \\ \cU \end{bmatrix} \to
\begin{bmatrix} \cX \\ \cY \end{bmatrix}, \qquad
M' = \begin{bmatrix} A' & B' \\ C' & D' \end{bmatrix} \colon \begin{bmatrix} \cX' \\ \cU \end{bmatrix} \to
\begin{bmatrix} \cX' \\ \cY \end{bmatrix}
$$
where $\cU$, $\cX$, $\cX'$, $\cY$ are all possibly infinite-dimensional Hilbert spaces.  Then:
\begin{enumerate}
\item[(1)]  Suppose that $\Sigma$ and $\Sigma'$ are both minimal, i.e., both are {\rm(}approximately{\rm)}
 controllable and {\rm(}approximately{\rm)} observable.  Then $\Sigma$ and $\Sigma'$ have transfer functions $F_\Sigma$
 and $F_{\Sigma'}$ agreeing on some neighborhood $\cN$ of the origin
 $$
   F_\Sigma(\lambda) = F_{\Sigma'}(\lambda) \text { for } \lambda \in \cN
 $$
 if and only if $\Sigma$ and $\Sigma'$ are {\em pseudo-similar} in the following sense
 {\rm(}see e.g.\ \cite[Section 3]{AKP06}{\rm)}:  $D= D'$ and there exists an injective,
closed, linear
operator $\Gamma \colon \cX  \to \cX'$ so that
\begin{equation}\label{pseudosim}
\begin{aligned}
    & \cD(\Gamma) \text{ is dense in } \cX, \quad \im \Gamma  \text{ is
dense in
    } \cX',  \\
    & A \cD(\Gamma) \subset \cD(\Gamma), \quad \Gamma A|_{\cD(\Gamma)} = A'
    \Gamma,   \\
 &B  \cU \subset \cD(\Gamma), \quad B' = \Gamma B,   \\
 & C|_{\cD(\Gamma)} = C' \Gamma.
\end{aligned}
\end{equation}

\item[(2)] Suppose that $\Sigma$ is exactly minimal while $\Sigma'$ is {\rm(}approximately{\rm)} minimal.  Then $F_\Sigma$ and $F_{\Sigma'}$ are identical on a neighborhood of $0$ if and only if $\Sigma$ and $\Sigma'$ are {\em similar}, i.e.,
there is a bounded and boundedly invertible linear operator $\Gamma \colon  \cX \to \cX'$ so that
\begin{equation}   \label{sys-sim}
  \begin{bmatrix} A' & B' \\ C' & D' \end{bmatrix} = \begin{bmatrix} \Gamma A \Gamma^{-1} & \Gamma B \\
  C \Gamma^{-1} & D \end{bmatrix}.
\end{equation}

\item[(3)]  Suppose that $\Sigma$ is $\ell^2$-exactly minimal, while $\Sigma'$ is {\rm(}approximately{\rm)} minimal and has
bounded controllability operator $\bW_c'$ as well as bounded observability operator $\bW_o'$.  Then $F_\Sigma$ and $F_{\Sigma'}$ are identical on a neighborhood of $0$ if and only if $\Sigma$ and $\Sigma'$ are similar as described in item (2).
\end{enumerate}
\end{theorem}

We note that item (1) in Theorem \ref{T:SSS} has been known for some time;
one can trace its origins to the work of Helton in \cite[Theorem 3.2]{Helton74} and of
Ball-Cohen \cite[Theorem 3.2]{BC}  with the fact that the pseudo-similarity can be taken
to be closed added later by Arov \cite{Arov79a}.  Indeed, this State-Space-Pseudo-Similarity
theorem is the main ingredient behind the proof of the first flavor of the infinite-dimensional
standard Bounded Real Lemma given above (item (1) in Theorem \ref{T:BRLinfstan}) in the work
of Arov-Kaashoek-Pik \cite{AKP06}.  Essentially the same proof can be used to prove items (2) and (3)
in Theorem \ref{T:BRLinfstan}, but with items (2) and (3) respectively from Theorem \ref{T:SSS}  (introduced
we believe here for the first time) used as the relevant State-Space-Similarity theorem
in place of item (1) from Theorem \ref{T:SSS}.

 For the reader's convenience
we include a complete, self-contained proof of part (1) of Theorem
\ref{T:SSS}, as the same framework applies to the proof of the new results, namely,
items (2) and (3) in Theorem \ref{T:SSS}.

Furthermore, with respect to statement (1) we mention that there
exist systems $\Si$  having transfer function $F_{\Sigma}$ in the Schur
class $\cS(\cU, \cY)$ such that every generalized positive-semidefinite solution  $H$ of the spatial KYP inequality
\eqref{KYP1b} is unbounded with unbounded inverse $H^{-1}$
(see \cite[Section 4.5]{AKP06}).

On the other hand, the strict Bounded Real Lemma extends to the
infinite-dimensional setting in essentially the same form as for the
finite-dimensional case.

\begin{theorem}[Infinite-dimensional strict Bounded Real Lemma]
\label{T:BRLinfstrict}
Let $\Si$ be a discrete-time linear system as in
\eqref{dtsystem} with system matrix $M$ as in \eqref{sysmat} and
transfer function $F_\Si$ defined by \eqref{trans}. Assume that $A$
is exponentially stable, i.e., $\spec(A) < 1$.  Then the transfer
function $F_{\Sigma}$ is in the strict Schur class $\cS^{o}(\cU,
\cY)$ if and only if there exists a bounded strictly positive-definite
solution $H$ of the  strict KYP-inequality \eqref{KYP2}.
\end{theorem}

This result is asserted in a number of papers in the engineering
literature, in particular in \cite[page 1490]{DL} where it is
attributed to Yakubovich \cite{Yak74, Yak75}; however it appears that
Yakubovich's stated result must be combined with some additional
(infinite-dimensional) inertia theorems to get the precise statement
here, namely that the operator $H$ is not only bounded selfadjoint but also
(invertible) positive-definite. The relatively recent paper of
Rantzer \cite{Rantzer} presents a new elementary proof using
convexity analysis for the finite-dimensional case.
 The infinite-dimensional version of the result appears
implicitly in the paper of Ben-Artzi--Gohberg--Kaashoek
\cite{BAGK95}, where the result is given in the more complicated
context of time-varying systems with dichotomy.

The original inspiration for the present paper was to resolve the
apparent discrepancies in these infinite-dimensional versions of the Bounded
Real Lemma, where in some instances it appears that unbounded operators are
required \cite{AKP06, AKP05},while in other instances one gets away with bounded
operators just as in the finite-dimensional
case \cite{DL, BAGK95}. A first inspection of Theorems
\ref{T:BRLinfstan} and \ref{T:BRLinfstrict}
suggests that this issue can be resolved by carefully distinguishing
between the {\em standard} and the {\em strict} Bounded Real Lemmas: one requires the
possibility of unbounded positive-definite solutions of the KYP-inequality in the standard
case but gets away with only bounded and boundedly invertible solutions in the strict
case.

Let us mention now how our results relate to a couple of other approaches which have appeared in the
literature:

\noindent
 \textbf{1.}  It is easy to see that a system $\Sigma'$ being {\em exactly minimal} implies that $\Sigma$
is in particular {\em {\rm(}approximately{\rm)} minimal}.  Hence in item (2) of Theorem \ref{T:SSS}, the result still
holds if the {\em minimality} assumption on $\Sigma'$ is replaced by a {\em exact minimality} hypothesis, in which
case the hypothesis in item (2) of Theorem \ref{T:SSS} assumes the more symmetric form:  assume that
{\em both $\Sigma$ and $\Sigma'$ are exactly minimal}.  Similarly, we shall see
as a consequence of the results in Section \ref{S:obscon} (specifically, Corollary \ref{C:ell2implics}
and items (4) and (9) in Proposition \ref{P:control-observe}) that
$\ell^2$-exact minimality implies boundedness of the associated observability and controllability operators
 $\bW_o$ and $\bW_c$ as well as (approximate) minimality.  Hence the result of item (3) in
 Theorem \ref{T:SSS} still holds if we impose the more symmetric assumption:  {\em
both $\Sigma$ and $\Sigma'$ are $\ell^2$-exactly minimal}.
In the language of  Chakhchoukh-Opmeer \cite{CO16}, the content of
items (2) and (3) in Theorem  \ref{T:SSS} is then that either of the two conditions (i) {\em exact
minimality} or (ii)  {\em $\ell^2$-exact minimality}
 gives a notion of {\em canonical realization} which leads to a good
 state-space isomorphism theorem, but with the caveat that not all
 transfer functions have such canonical realizations.  The approach
of \cite{CO16} to a good state-space isomorphism theorem, on the other hand, is to
 extend the category where state spaces are to reside from Hilbert
 spaces to locally convex topological vector spaces which are
 Hausdorff and barrelled, and then to assume that the given systems
are minimal in the sense that neither has a nontrivial Kalman reduction.

\smallskip

 \noindent
\textbf{2.} Willems \cite{Wil72a, Wil72b} has given an energy-dissipation
interpretation of positive-definite solutions $H$ of the KYP
inequality as follows. We view the function $S:x \mapsto \|x\|_H^2: =
\langle Hx, x \rangle_{\cX}$ as a measure of energy stored by the
state $x$ in the state space $\cX$. The KYP-inequality \eqref{KYP1}
can then be rewritten as
\begin{equation}   \label{storage}
  S(\bx(n+1)) - S(\bx(n)) \le \|\bu(n)\|^{2} - \| \by(n) \|^{2}
\end{equation}
which should hold for any system trajectory $(\bu(n), \bx(n),
\by(n))_{n  \in {\mathbb Z}}$ of \eqref{dtsystem}.
Let us say that a function $S \colon \cX \to {\mathbb R}_{+}$ is a {\em
storage function} for the system $\Sigma$ if the energy balance
relation \eqref{storage} holds over all trajectories of the system,
subject to the additional normalization condition
\begin{equation}   \label{normalization}
\min_{x \in \cX} S(x) = S(0) = 0.
\end{equation}
In words this says: {\em The net energy stored by the system in the
transition from state $\bx(n)$ to $\bx(n+1)$ is no more than the net
energy supplied to the system from the outside environment, as
measured by the supply rate $s(\bu(n), \by(n)) = \|\bu(n)\|^{2} -
\|\by(n)\|^{2}$.}  In a forthcoming report \cite{KYP2}, we show how to
arrive at the infinite-dimensional Bounded Real Lemmas as presented
here via explicit computation of extremal Willems storage functions, rather
than via application of  infinite-dimensional State-Space-Similarity theorems
as is done here.

\smallskip

 The paper is organized as follows.  After the current Introduction, Section~\ref{S:obscon} develops more precise
 statements concerning observability operators $\bW_o$ and controllability operators $\bW_c$ needed in the sequel
 for the general unbounded setting.
 Section \ref{S:SSS} proves the new parts (2) and (3)
 of the infinite-dimensional State-Space-Similarity theorem (Theorem \ref{T:SSS}), as well as sketches the
 proof of part (1) needed as the framework of the proofs of (2) and (3).
 Section \ref{S:infstanBRL} goes through the three flavors of the infinite-dimensional standard Bounded Real Lemma
 (Theorem \ref{T:BRLinfstan}) while the final section (Section \ref{S:infstrictBRL}) shows how the regularization technique
 of Petersen-Anderson-Jonckheere \cite{PAJ} can be adapted to this infinite-di\-mensional setting to give a proof of the
 infinite-dimensional strict Bounded Real Lemma (Theorem \ref{T:BRLinfstrict}).

\section{The observability and controllability operators}
\label{S:obscon}

In this section we introduce the observability and controllability operators associated with the discrete-time linear system $\Si$ given by \eqref{dtsystem} and derive some of their basic properties. For the case of a general system $\Sigma$, we define the {\em observability operator} $\bW_{o}$ associated with $\Si$ to be the possibly unbounded operator with domain $\cD(\bW_{o})$ in $\cX$ given by
\begin{equation} \label{bWo1}
\cD(\bW_{o}) = \{ x \in \cX \colon \{ C A^{n} x\}_{n \ge 0}
\in\ell^{2}_{\cY}({\mathbb Z}_{+})\}
\end{equation}
with action given by
\begin{equation}   \label{bWo2}
\bW_{o} x =  \{ C A^{n}  x\}_{n \ge 0} \text{ for } x \in
\cD(\bW_{o}).
\end{equation}
Dually, we define the {\em adjoint controllability operator} $\bW_{c}^{*}$ associated with $\Si$ to have domain
\begin{equation}   \label{bWc*1}
\cD(\bW_{c}^{*}) = \{ x \in \cX \colon \{B^* A^{*(-n-1)} x\}_{n\le
-1} \in\ell^{2}_{\cU}({\mathbb Z}_{-})\}
\end{equation}
with action given by
\begin{equation}   \label{bWc*2}
\bW_{c}^{*} x = \{B^* A^{*(-n-1)} x\}_{n\le -1} \text{ for } x \in
\cD(\bW_{c}^{*}).
\end{equation}
It can happen that $\cD(\bW_o) = \{0\}$ (e.g., $\cY = \cX = {\mathbb C}$ with $C=1$,
$A=2$), and similarly for $\cD(\bW_c^*)$.  Nevertheless, both $\bW_o$ and $\bW_c^*$ are always
closed operators, and, when it is the case that their domains are dense and hence they have adjoints,
the adjoints are explicitly computable, as recorded in the next result.

\begin{proposition}  \label{P:WcWo'}
Let $\Sigma$ be a system as in \eqref{dtsystem} with observability operator $\bW_o$ and adjoint controllability
operator $\bW_c^*$ as in \eqref{bWo1}--\eqref{bWc*2}.  Then:
\begin{enumerate}
\item[(1)] $\bW_o$ is a closed operator on its domain \eqref{bWo1}.

\item[(2)] Assume that $\cD(\bW_o)$ is dense in $\cX$.  Then
the adjoint $\bW_o^*$ of $\bW_o$ is a closed,
densely defined operator with domain
$\cD(\bW_o^*)$ containing the linear manifold
$\ell_{\tu{fin},\cY}(\BZ_+)$ of finitely supported sequences in
$\ell^2_\cY(\BZ_+)$.  In general, $\cD(\bW_{o}^{*})$ is characterized
as the set of all $\by \in \ell^{2}_{\cY}({\mathbb Z}_{+})$ such that
there exists a vector $x_{o} \in \cX$ such that the limit
$\lim_{K \to \infty}\langle  x, \sum_{k=0}^{K} A^{*k} C^{*} \by(k)
 \rangle_{\cX}$ exists for each $x \in \cD(\bW_o)$ and is given  by
\begin{equation}   \label{limit-o}
 \lim_{K \to \infty}\langle  x, \sum_{k=0}^{K} A^{*k} C^{*} \by(k)
 \rangle_{\cX} = \langle x, x_{o} \rangle_{\cX},
\end{equation}
and then the action of
$\bW_{o}^{*}$ is given by
\begin{equation}   \label{Wo*act}
    \bW_{o}^{*} \by = x_{o}
 \end{equation}
 where $x_{o}$ is as in \eqref{limit-o}.
  In particular, $\ell_{\tu{fin}, \cY}({\mathbb Z}_+)$ is contained in $\cD(\bW_o^*)$ and the observability
space defined in \eqref{ObsSpace} is given by
$$
\Obs (C|A) = \bW_{o}^{*} \ell_{\tu{fin}, \cY}({\mathbb Z}_{+}).
$$
Thus, if in addition $(C,A)$ is observable, then $\bW_o^*$ has dense range.

\item[(3)] The adjoint controllability operator $\bW_c^*$ is closed on its domain \eqref{bWc*1}.

\item[(4)]  Assume $\cD(\bW_c^*)$ is dense in $\cX$.  Then
the controllability operator $\bW_c$ defined as the adjoint of $\bW_c^*$ is a closed,
 densely defined operator with domain $\cD(\bW_c)$ containing the linear manifold
$\ell_{\tu{fin},\cU}(\BZ_-)$ of finitely supported sequences in
$\ell^2_\cU(\BZ_-)$.  In general, $\cD(\bW_{c})$ is characterized as
the set of all $\bu \in \ell^{2}_{\cU}({\mathbb Z}_{-})$ such that
there exists a vector $x_{c} \in \cX$ so that
$ \lim_{K \to \infty} \langle x, \sum_{k=-K}^{-1} A^{-k-1} B \bu(k)
    \rangle_{\cX} $ exists for each $x \in \cD(\bW_{c}^{*})$ and is given by
\begin{equation}  \label{limit-c}
\lim_{K \to \infty} \langle x, \sum_{k=-K}^{-1} A^{-k-1} B \bu(k)
    \rangle_{\cX} = \langle x, x_{c} \rangle_{\cX},
\end{equation}
and then the action of $\bW_{c}$ is given by
\begin{equation}  \label{Wc-act}
    \bW_{c} \bu = x_{c}
\end{equation}
where $x_{c}$ is as in \eqref{limit-c}.
 In particular, the reachability space
$\Rea (A|B)$ is equal to $\bW_{c} \ell_{{\rm fin}, \cU}({\mathbb
Z}_{-})$.
Thus, if in addition $(A,B)$ is controllable, then $\bW_c$ has dense range.
\end{enumerate}
\end{proposition}

\begin{proof}  Note that once (1) and (2) are verified, (3) and (4) follow directly by applying
(1) and (2)  to the adjoint system $\Si^*$ defined by
\begin{equation}\label{dtsystem*}
\Si^{*}:=\left\{
\begin{array}{ccc}
\bx(n)&=&A^{*}\bx(n+1)+C^{*}\by(n),\\
\bu(n)&=&B^{*}\bx(n+1)+D^{*}\by(n),
\end{array}
\right. \qquad (n\in\BZ).
\end{equation}
Hence it suffices to prove (1) and (2).

To show that $\bW_{o}$ is closed, we must show:  {\em Whenever
$\{x_{k}\}_{k \in {\mathbb Z}_{+}}$ is a sequence of vectors in
$\cD(\bW_{o})$ converging to a vector $x \in \cX$ such that the
output sequence $\by_{k} = \bW_{o} x_{k}$ converges to a vector $\by
\in \ell^{2}_{\cY}({\mathbb Z}_{+})$, then it follows that $x \in
\cD(\bW_{o})$ and $\bW_{o} x = \by$.}
We therefore assume that we have a sequence of vectors $\{x_{k}\}_{k
\ge 0}$ from $\cD(\bW_{o})$ with $\lim_{k\to \infty} x_{k} = x$ in
$\cX$ and $\lim_{k \to \infty} \bW_{o} x_{k} = \by$ in
$\ell^{2}_{\cY}({\mathbb Z}_{+})$.
Fix $n\in\BZ_+$. From the assumption that $\lim_{k \to \infty} x_{k}
\to x$ in $\cX$, since $C$ and $A$ are bounded operators, it follows
that
\begin{equation}  \label{lim1}
\lim_{k \to \infty} C  A^{n} x_{k} = C A^{n} x\text{ in } \cY.
\end{equation}
On the other hand, continuity of the evaluation map $\boldsymbol{\rm
ev}_{n} \colon \ell^{2}_{\cY}({\mathbb Z}_{+})\to \cY$ given by
$\boldsymbol{\rm ev}_{n} \colon \by \mapsto \by(n)$ implies that
\[
\lim_{k\to\infty} CA^n x_k
=\lim_{k\to\infty} \boldsymbol{\rm ev}_{n}\bW_o x_k
=\boldsymbol{\rm ev}_{n} \lim_{k\to\infty} \bW_o x_k
=\boldsymbol{\rm ev}_{n}\by=\by(n)
\]
for each nonnegative integer $n$.
Thus $CA^n x=\by(n)$ holds for each $n\in\BZ_+$. This implies that
$\{CA^n x\}_{n\in\BZ_+}$ is in $\ell^2_\cY(\BZ_+)$, hence
$x\in\cD(\bW_o)$, and $\bW_o x=\{CA^n
x\}_{n\in\BZ_+}=\{\by(n)\}_{n\in\BZ_+}=\by$.
Thus $\bW_o$ is a closed operator and (1) follows.

Let us now assume that $\cD(\bW_o)$ is dense. As we have shown that $\bW_o$ is closed,
it follows that $\bW_o$ is adjointable with adjoint $\bW_o^*$ also closed and densely defined (see
Theorem VIII.1 of \cite{RS}). For the particular case of $\bW_o^*$ here, we show that in fact $\cD(\bW_o^*)$ contains the dense linear manifold $\ell_{\tu{fin}, \cY}({\mathbb Z}_+)$. Let $\by \in
\ell_{{\rm fin}, \cY}({\mathbb Z}_{+})$. Define $x_{o} \in \cX$ by
the finite sum $x_{o} = \sum_{n \in {\mathbb Z}} A^{*n} C^{*}
\by(n)$. Then for each  $x \in \cD(\bW_{o})$ we have
\[
\langle \bW_{o} x, \by\rangle_{\ell^{2}_{\cY}({\mathbb Z}_{+})}
=\sum_{n \in {\mathbb Z}_{+}} \langle C A^{n} x, \by(n) \rangle_{\cY}
= \sum_{n \in {\mathbb Z}_{+}} \langle x, A^{n*} C^{*}
\by(n)\rangle_{\cX}
= \langle x, x_{o} \rangle_{\cX}.
\]
This shows that $\by \in \cD(\bW_{o}^{*})$ with $\bW_{o}^{*}\by =
x_{o}$. We obtain that $\ell_{{\rm fin}, \cY}({\mathbb Z}_{+})$  is a
subset of  $\cD(\bW_{o}^*)$. Since $\ell_{{\rm fin}, \cY}({\mathbb
Z}_{+})$ is dense in $\ell^2_\cY(\BZ_+)$, so is $\cD(\bW_{o}^*)$.

More generally, suppose that $\by \in \ell^{2}_{\cY}({\mathbb
Z}_{+})$ is such that there exists a vector $x_{o} \in \cX$ so that
\eqref{limit-o} holds for all $x \in \cD(\bW_{o})$.  Then, for $x \in
\cD(\bW_{o})$ we have
$$
\langle \bW_{o} x, \by \rangle_{\ell^{2}_{\cY}({\mathbb Z}_{+})} =
\sum_{k=0}^{\infty} \langle C A^{k} x, \by(k) \rangle_{\cY} = \lim_{K
\to \infty} \langle x, \sum_{k=0}^{K} A^{*k} C^{*} \by(k)
\rangle_{\cX}  = \langle x, x_{o} \rangle_{\cX}
$$
and we conclude that $\by \in \cD(\bW_{o}^{*})$ with $\bW_{o}^{*} \by
= x_{o}$.  Conversely, suppose that $\by \in \cD(\bW_{o}^{*})$ with
$\bW_{o}^{*} \by = x_{o}$.  Then, for $x \in \cD(\bW_{o})$ we have
\begin{align*}
\langle x, x_{o} \rangle_{\cX}
&= \langle x, \bW_{o}^{*} \by\rangle_{\cX}
= \langle \bW_{o}x, \by \rangle_{\ell^{2}_{\cY}({\mathbb Z}_{+})}\\
&= \sum_{k=0}^{\infty} \langle C A^{k} x, \by(k)
\rangle_{\cY} = \lim_{K \to \infty} \langle x, \sum_{k=0}^{K} A^{*k}
C^{*} \by(k) \rangle_{\cX}
\end{align*}
and we conclude that the pair $\by, x_{o}$ is as in \eqref{limit-o}.
We have now verified the claimed characterization \eqref{limit-o}--\eqref{Wo*act} of $\cD(W_{o}^{*})$.

From this characterization we read off that $\im \bW_o^*  \supset \Obs(C|A)$.  Hence if we assume in addition that
$(C,A)$ is observable, we conclude that $\bW_o^*$ has dense range and (2) follows.
\end{proof}

Much more can be said about the observability and (adjoint) controllability operators in case the transfer function $F_\Si$ of $\Si$, given by \eqref{trans}, has an analytic continuation to a function in $H^\infty(\cU,\cY)$. We first collect a few observations about this case.

\begin{proposition}   \label{P:Hankel}
Suppose that $F(z) = \sum_{n=0}^{\infty} F_{n} z^{n}$ defines an
$H^{\infty}$-function on ${\mathbb D}$.   Then the following
statements hold:
\begin{enumerate}
    \item[(1)] The Hankel matrix
$$  \fH_{F} = [F_{i-j}]_{i\ge 0,\, j<0}
$$
defines a bounded operator from $\ell^{2}_{\cU}({\mathbb Z}_{-})$
into $\ell^{2}_{\cY}({\mathbb Z}_{+})$.

\item[(2)]
In case $F=F_\Si$ is the transfer function from a system $\Si$ as in
\eqref{dtsystem}
with system matrix $\sbm{A & B \\ C & D}$, then the Hankel matrix
$\fH_{F_{\Sigma}}$
is given by
\begin{equation}   \label{HankelSigma}
   \fH_{F_{\Sigma}}  = [C A^{i-j-1} B]_{i \ge 0,\, j<0}.
\end{equation}

\end{enumerate}
\end{proposition}

\begin{proof}
Suppose that the function $F$ is a $\cL(\cU, \cY)$-valued
$H^{\infty}$-function on ${\mathbb D}$.
Then the multiplication operator
\begin{equation}  \label{multiplicationop}
M_{F} \colon u(z) \mapsto F(z) u(z)
\end{equation}
is bounded as an operator from $H^{2}_{\cU}({\mathbb D})$ to
$H^{2}_{\cY}({\mathbb D})$.  Moreover, the same formula
\eqref{multiplicationop} can be used to extend $M_{F}$ to a
a bounded operator from $L^{2}_{\cU}({\mathbb T})$ to
$L^{2}_{\cY}({\mathbb T})$ of the same norm, and this operator is
contractive in case $S \in \cS(\cU, \cY)$.  Let us define the {\em
Hankel
operator} ${\mathbb H}_{F} \colon H^{2}_{\cU}({\mathbb D})^{\perp}
\to H^{2}_{\cY}({\mathbb D})$ by
$$
  {\mathbb H}_{F} = P_{H^{2}_{\cY}({\mathbb D})}
  M_{F_{\Sigma}}|_{H^{2}_{\cY}({\mathbb D})^{\perp}}.
$$
As $ M_F$ is bounded (contractive in case $F \in \cS(\cU, \cY)$), it
follows that
also ${\mathbb H}_{F} $ is bounded (contractive in case $F \in
\cS(\cU, \cY)$).  We shall be interested in the inverse-$Z$-transform
version of these observations.

The inverse-$Z$-transform version of $M_{F}$ is given by the
biinfinite Laurent operator
\begin{equation}\label{Laurent'}
   \frakL_{F_\Si} = [ F_{i-j}]_{-\infty < i,j <
\infty}:\ell^2_\cU(\BZ)\to\ell^2_\cY(\BZ)
\end{equation}
where the Taylor series $F(z) = \sum_{n=0}^{\infty} F_{n} z^{n}$ for
$F$ determines $F_{n}$ for $n \ge 0$ and where we set $F_{n} = 0$ for
$n <
0$.  By the unitary property of the $Z$-transform from
$\ell^{2}({\mathbb Z})$ to
$L^{2}({\mathbb T})$, we see that $\fL_{F}$ has the same norm as
$M_{F}$ and hence is bounded (contractive in case $F \in \cS(\cU,
\cY)$) as an operator from $\ell^{2}_{\cU}({\mathbb Z})$ to
$\ell^{2}_{\cY}({\mathbb Z})$.   The inverse-$Z$-transform version of
the Hankel operator ${\mathbb H}_{F}$ is the time-domain version of
the Hankel operator $\fH_{F} \colon
\ell^{2}_{\cU}({\mathbb Z}_{-}) \to \ell^{2}_{\cY}({\mathbb Z}_{+})$
with matrix representation given by the southwest corner of the
Laurent matrix $\fL_{F}$:
$$
 \fH_{F} = [F_{i-j}]_{i \ge 0,\, j < 0}
$$
and trivially has norm bounded by the norm of $\fL_{F}$.
This completes the verification of statement (1) of
the proposition.

In case $F = F_{\Sigma}$ is the transfer function of a system as in
\eqref{dtsystem}
with system matrix  $\sbm{ A & B \\ C & D }$, so
$$
 F_{\Sigma}(z) = D + z C (I - zA)^{-1} B,
$$
then clearly
$F_{0} =  D$,
$F_{n} = C A^{n-1} B$ if  $n \ge 1$, and
$F_{n} = 0$ if  $n<0$,
%
and the Hankel matrix has the form
$$
  \fH_{F_{\Sigma}} = [ C A^{i-j-1}B]_{i \ge 0,\, j<0}
$$
and statement (2) of the proposition follows.
\end{proof}

The next proposition shows, among others, that the denseness conditions in items (2) and (4) of Proposition \ref{P:WcWo'} are automatically satisfied if $F_\Si$ has an analytic continuation to an $H^\infty$ function.

\begin{proposition}\label{P:WcWo}
Let $\Si$ be a discrete-time linear system as in
\eqref{dtsystem} with system matrix $M$ as in \eqref{sysmat}.
Assume that the transfer function $F_\Si$ defined by \eqref{trans}
has an analytic continuation to an $\cL(\cU, \cY)$-valued
$H^{\infty}$-function on ${\mathbb D}$.  Define $\bW_{o}$ and
$\bW_{c}^{*}$ as in
\eqref{bWo1}--\eqref{bWo2} and \eqref{bWc*1}--\eqref{bWc*2},
respectively. Then:
\begin{enumerate}
\item[(1)] The domain $\cD(\bW_o)$ of $\bW_o$ contains the reachability subspace $\Rea(A|B)$.
Thus, if $(A,B)$ is controllable, then $\cD(\bW_o)$ is dense in $\cX$.  If in addition $(C,A)$ is
observable, then $\bW_o$ is injective.

\item[(2)]  The domain $\cD(\bW_c^*)$ of the adjoint controllability operator $\bW_c^*$
contains the observability space $\Obs(C|A)$.  Hence, if
$(C,A)$ is observable, then $\cD(\bW_c^*)$ is dense in $\cX$.
If in addition $(A,B)$ is controllable, then $\bW_c^*$ is injective.
\end{enumerate}
\end{proposition}

\begin{proof}  Note that (2) follows from (1) by duality, so it suffices to consider (1).

We first show that $\cD(\bW_o)$ contains the reachability space.
 For this purpose, note that since $F_\Si$ is in $H^\infty$,
we know by Proposition \ref{P:Hankel} that the Hankel operator
$\fH_{F_\Si}$ defined by $F_\Si$ is a bounded operator from
$\ell^2_\cU(\BZ_-)$ into $\ell^2_\cY(\BZ_+)$. Let $\bu\in
\ell_{\tu{fin},\cU}(\BZ_-)$, say $\bu$ has support in the entries
indexed with $K,\ldots, -1$ and $\bu=\{\bu(k)\}_{k=K}^{-1}$. Then as a
consequence of the matrix representation \eqref{HankelSigma} for
$\fH_{F_{\Sigma}}$, we see that the action of $\fH_{F_{\Sigma}}$ on
$\bu$ can be arranged to have the form
\[
\fH_{F_\Si}\bu=\bW_o\left(\sum_{k=K}^{-1} A^{-1-k}B \bu(k)\right)\in
\ell^2_\cY(\BZ_+).
\]
Since $\bu\in \ell_{\tu{fin},\cU}(\BZ_-)$ was chosen arbitrarily it
follows that $\cD(\bW_o)$ contains all vectors from the reachability
space $\Rea (A|B)$. If we assume that $(A, B)$ is controllable, it then follows
that $\cD(\bW_o)$ is dense in $\cX$.

If $x \in \cD(\bW_o)$ is such that $\bW_o x = 0$, then $C A^n x = 0$ for all $n \in
{\mathbb Z}_+$.  If we assume that $(C,A)$ is observable, it now follows that $x = 0$,
i.e., it follows that $\bW_o$ is injective. This completes the proof of (1).
\end{proof}

The precise characterizations of $\cD(\bW_{o}^{*})$ and $\cD(\bW_{c})$
in Proposition \ref{P:WcWo'} enables us to pick up the following
useful corollary.  Recall  that
an operator $T$ is said to be {\em bounded below} in case there
exists a $\de>0$ so that
\begin{equation}  \label{boundedbelow}
\|T x\|\geq \de \|x\| \text{ for all } x\in\cD(T).
\end{equation}
We note that if $T$ is positive-definite, then $T$ being bounded below is equivalent to
$T^{-1}$ being bounded.

\begin{corollary} \label{C:HankDecs}
Assume that we are given a discrete-time linear system
\eqref{dtsystem} with associated transfer function $F_\Si$ {\rm(}possibly
after analytic continuation{\rm)} equal to an $H^{\infty}$-function on
${\mathbb D}$,
observability operator $\bW_{o}$ and adjoint controllability operator
$\bW_{c}^{*}$, and Hankel matrix $\fH_{F_{\Sigma}}$ as in
\eqref{trans}, \eqref{bWo1}, \eqref{bWo2}, \eqref{bWc*1},
\eqref{bWc*2}, \eqref{HankelSigma}.  Then:
\begin{enumerate}
    \item[(1)]
    Assume that $\cD(\bW_c^*)$ is dense in $\cX$ {\rm(}alternatively, by Proposition \ref{P:WcWo},
    assume that $(C,A)$ is observable{\rm)}.  Then
    $\cD(\bW_{o})$ contains $\im \bW_{c} = \bW_{c} \cD(\bW_{c})$ and
\begin{equation}  \label{HankDec1}
\fH_{F_{\Sigma}}|_{\cD(\bW_{c})} = \bW_{o} \bW_{c}.
\end{equation}
\item[(2)]
Assume that $\cD(\bW_o)$ is dense in $\cX$ {\rm(}alternatively, by Proposition \ref{P:WcWo}, assume that
$(A,B)$ is controllable{\rm)}.  Then
$\cD(\bW_{c}^{*})$ contains $\im \bW_{o}^{*} = \bW_{o}^{*}\cD(\bW_{o}^{*})$ and
\begin{equation}   \label{HankDec2}
\fH_{F_{\Sigma}}^{*}|_{\cD(\bW_{o}^{*})} = \bW_{c}^{*} \bW_{o}^{*}.
\end{equation}
\end{enumerate}
\end{corollary}

\begin{proof}
Suppose that $\bu \in \cD(\bW_{c})$, a dense subset of $\ell^2_\cU({\mathbb Z}_-)$
by Proposition \ref{P:WcWo'}.  Then from formula \eqref{limit-c} we see that
$x_{c} = \bW_{c} \bu$
    is determined by
 \begin{equation}   \label{limit-c'}
 \lim_{K \to \infty} \langle x, \bW_{c} \bu_{K} \rangle_{\cX} =
 \langle x, x_{c} \rangle_{\cX}
 \end{equation}
 for each $x$ in the dense (by assumption) subset $ \cD(\bW_{c}^{*})$, where we set
 $$
  \bu_{K}(k) = \begin{cases} \bu(k) &\text{if } k \ge  -K, \\
      0 &\text{otherwise.}
      \end{cases}
 $$
 In particular, by Proposition \ref{P:WcWo} we know that $\Obs(C|A)
 \subset \cD(\bW_{c}^{*})$ and hence, for any $y \in \cY$,
\eqref{limit-c'} holds with
 $A^{*n} C^{*} y$ in place of $x$.  This then leads us to
 $\lim_{K \to \infty}\langle y, C A^{n} \bW_{c} \bu_{K} \rangle_{\cY}
=
 \langle y, C A^{n} x_{c} \rangle_{\cY}$ for each $y \in \cY$,
 i.e., to
 \begin{equation}   \label{limit-c''}
    \underset{K \to \infty}{\operatorname{weak-lim}} \,  C A^{n}
\bW_{c}
    \bu_{K} = C A^{n} x_{c}.
 \end{equation}
 Note that $C A^{n} \bW_{c} \bu_{K} =
 \boldsymbol{\rm ev}_{n} \fH_{F_{\Sigma}} \bu_{K}$.  As
 $\fH_{F_{\Sigma}}$ is bounded  and
 $\lim_{K \to \infty} \bu_{K} = \bu$ in norm, it
 follows that $\lim_{K \to \infty} \boldsymbol{\rm ev}_{n}
 \fH_{F_{\Sigma}} \bu_{K} =  \boldsymbol{\rm ev}_{n}
\fH_{F_{\Sigma}} \bu$
 in the norm topology  of $\cY$ for each $n$.   On the other hand,
from \eqref{limit-c''}
 we see that $\lim_{K \to \infty} C A^{n} \bW_{c} \bu_{K} = C A^{n}
x_{c}$ in the weak
 topology of $\cY$. As norm convergence implies weak convergence,
 uniqueness of weak limits implies the equality $C A^{n} x_{c} =
 \boldsymbol{\rm ev}_{n} \fH_{F_{\Sigma}} \bu$.
  As this holds for all $n=0,1,2,\dots$, we
conclude that $\{ C A^{n} x_{c}\}_{n \ge 0} = \fH_{F_{\Sigma}} \bu$
is in
$\ell^{2}_{\cY}({\mathbb Z}_{+})$, i.e., $x_{c} \in \cD(\bW_{o})$ and
$\bW_{o} \bW_{c} \bu = \bW_{o} x_{c} = \fH_{F_{\Sigma}} \bu$.

The assertion for $\bW_{c}^{*} \bW_{o}^{*}$ follows by a dual
analysis.
\end{proof}

The next corollary list some useful consequences of $\ell^2$-exact controllability
and $\ell^2$-exact observability.

\begin{corollary}\label{C:ell2implics}
Let $\Si$ be a  discrete-time linear system as in
\eqref{dtsystem} with system matrix
$M$ as in \eqref{sysmat}. Assume that the transfer function $F_\Si$
defined by \eqref{trans}
has an analytic continuation to an $\cL(\cU, \cY)$-valued
$H^{\infty}$-function on ${\mathbb D}$.
\begin{itemize}
\item[(1)] If $\Sigma$ is $\ell^2$-exactly controllable, then $\bW_o$
is bounded.

\item[(2)] If $\Sigma$ is $\ell^2$-exactly observable, then $\bW_c$
is bounded.

\item[(3)] $\Sigma$ is $\ell^2$-exactly minimal, i.e., both
$\ell^2$-exactly controllable and $\ell^2$-exactly observable,
then $\bW_o$ and $\bW_c^*$ are both bounded and bounded below.

\end{itemize}
\end{corollary}

\begin{proof}
Assume that $\Sigma$ is $\ell^2$-exactly controllable.  In particular,
by the definition of $\ell^2$-exact controllability in the Introduction, $\cD(W_c^*)$ is
dense in $\cX$. Then the $\ell^2$-exact controllability hypothesis combined with item (1)
in Corollary  \ref{C:HankDecs} tells us that $\cD(\bW_{o})$ is the whole space
 $\cX$.  As $\bW_{o}$ is a closed operator (as verified in Proposition \ref{P:WcWo'}),
 it follows from the Closed Graph
 Theorem that $\bW_{o}$ is bounded. This verifies item (1) in
 Corollary \ref{C:ell2implics}.  Item (2) in Corollary
 \ref{C:ell2implics} follows  by the dual analysis.

 Next suppose that $\Sigma$ is $\ell^{2}$-exactly minimal, so we know
 that $\bW_{o}$ and $\bW_{c}$ are bounded by items (1) and (2) above.
  The $\ell^{2}$-exact-minimality hypothesis gives us that
$\im \bW_o^*=\cX$ and $\im \bW_c=\cX$.  Hence also $\bW_o^* \bW_o$ and
$\bW_c \bW_c^*$ are surjective.  From the fact that $\bW_o^*$ and $\bW_c$ are surjective,
it follows that $\bW_o$ and $\bW_c^*$ are injective, and hence also
$\bW_o^* \bW_o$ and $\bW_c \bW_c^*$ are injective.  It now follows from the Open Mapping Theorem
that $\bW_o^* \bW_o$ and $\bW_c \bW_c^*$ are bounded below, and hence also $\bW_o$ and $\bW_c^*$
are bounded below.
\end{proof}

A well-known case in which $\bW_o$ and $\bW_c$ are bounded is when the system matrix $M$ in \eqref{sysmat} is a contraction. In this case, as mentioned in the Introduction (see also \cite{HZ}),  the transfer function $F_\Si$ is a Schur class function.
For later use we record the following result.

\begin{proposition}\label{P:contractiveSM}
Let $\Si$ be the discrete-time linear system \eqref{dtsystem} with transfer function $F_\Si$ given by \eqref{trans}. Assume that the system matrix $M$ in \eqref{sysmat} is a contraction. Then $F_\Si$ is in the Schur class $\cS(\cU,\cY)$ and the controllability operator $\bW_{c}$ and observability operator $\bW_{o}$ are contraction operators with respective row- and column-matrix representations
$$
  \bW_{c} = \operatorname{row}_{j<0} [A^{-j-1} B]: \ell^{2}_{\cU}({\mathbb
Z}_{-}) \to \cX, \quad
  \bW_{o} = \operatorname{col}_{i\ge 0} [C A^{i}]: \cX \to \ell^{2}_{\cY}({\mathbb
Z}_{+})
$$
and furthermore provide a factorization of the Hankel operator $\fH_{F_{\Sigma}}$:
$$
  \fH_{F_{\Sigma}} = \bW_{o} \bW_{c}.
$$
\end{proposition}

\begin{proof}
In case the system matrix $\Sigma = \sbm{ A & B \\ C & D}$ is
contractive, then in particular the  row matrix $\begin{bmatrix} A &
B \end{bmatrix}$ is contractive so we have $A A^{*} + B B^{*} \preceq
I$.  Hence we have
\begin{align*}
    & \begin{bmatrix} A^{N}B & \cdots & AB & B \end{bmatrix}
\begin{bmatrix} B^{*} A^{*N} \\ \vdots \\ B^{*} A^{*} \\ B^{*}
\end{bmatrix} = \sum_{k=0}^{N} A^{k} B B^{*} A^{*k} \\
& \qquad\qquad \preceq
\sum_{k=0}^{N} A^{k}(I - A A^{*}) A^{*k} = I - A^{N+1} A^{* N+1}
\preceq I
\end{align*}
and hence
$$
\bW_{c} \bW_{c}^{*} = \operatorname{s-lim}_{N \to \infty}
\sum_{k=0}^{N} A^{k} B B^{*} A^{*k} \preceq I
$$
and it follows that $\| \bW_{c} \| \le 1$.  The proof that
$\bW_{o}^{*}
\bW_{o} \preceq I$ proceeds similarly making use of the fact that
$A^{*}A + C^{*} C \preceq I$, and statement (3) of the Proposition
follows.  As  observed in the Introduction, the result of \cite{HZ} tells us
(even for the nonrational case) that $F_\Sigma$
is a Schur-class function when $\| M \| \le 1$.
\end{proof}

We are now in position to sort out the connections among the notions
of controllable/exactly controllable/$\ell^{2}$-exactly controllable
and the dual notions of observable/exactly
observable/$\ell^{2}$-exactly observable.

\begin{proposition} \label{P:control-observe}
    Suppose that $\Sigma$ is a linear system with system matrix $M =
    \sbm{ A & B \\ C & D }$ as in \eqref{sysmat}.
    \begin{itemize}
\item[(1)]  It can happen that $(A,B)$ is exactly controllable but
not $\ell^{2}$-exactly controllable.

\item[(2)] It can happen that $(A,B)$ is $\ell^{2}$-exactly
controllable but not exactly controllable.

\item[(3)] If $(A,B)$ is exactly controllable, then $(A,B)$
is controllable.

\item[(4)] If $(A,B)$ is $\ell^{2}$-exactly controllable with $\cD(\bW_c^*) =\cX$
{\rm(}so $\bW_c^*$ is bounded and $\bW_c$ is not only bounded but also
surjective{\rm)}, then $(A,B)$ is controllable.

\item[(5)] If $(A, B)$ is exactly controllable and $\cD(\bW_c^*)$ is dense,
 then $(A,B)$ is $\ell^{2}$-exactly controllable.

\item[(6)]  It can happen that $(C,A)$ is exactly observable but not
$\ell^{2}$-exactly observable.

\item[(7)] It can happen that $(C,A)$ is $\ell^{2}$-exactly
observable but not exactly observable.

\item[(8)] If $(C,A)$ is exactly observable, then $(C,A)$ is
observable.

\item[(9)] If $(C,A)$ is $\ell^{2}$-exactly observable and
$\cD(\bW_o) = \cX$ {\rm(}so $\bW_o$ is bounded and $\bW_o^*$ is
not only bounded but also surjective{\rm)}, then $(C,A)$ is observable.

\item[(10)] If $(C,A)$ is exactly observable and $\cD(\bW_o)$ is dense,
then $(C,A)$ is $\ell^{2}$-exactly observable.
 \end{itemize}
\end{proposition}

\begin{proof}
As items (6)--(10) are just dual versions of
items  (1)--(5), we need only prove (1)--(5).

\smallskip

\paragraph{\bf (1):}
    Take $\cX = \cU = {\mathbb C}$ with $A = [2]$ and $B = [1]$. Then $\bW_{c}^{*}$ defined by
    \eqref{bWc*1}--\eqref{bWc*2} has domain $\cD(\bW_{c}^{*})$ equal
    to the zero space, so in particular is not dense in $\cX$. Then, according to our definition,
    $(A,B)$ is not $\ell^2$-exactly controllable.   Nevertheless it is
   clear that the
    input pair $(A,B)$ is exactly controllable.

    To remedy this situation we may attempt instead to use the formulas
    \eqref{limit-c}--\eqref{Wc-act} to define a controllability
    operator $\widetilde \bW_{c}$; however, for our example $A=[2]$,
$B = [1]$, the resulting $\widetilde \bW_{c}$ is not closed or even
    closable.  Note if we choose $C$ so that $(C,A)$ is observable
    (e.g., $C = [1]$), the resulting transfer function $F_{\Sigma}(z)
    = \frac{z}{1 - 2 z}$ does not have analytic continuation to an
    $H^{\infty}$-function on ${\mathbb D}$.  This example illustrates
    the crucial role of the hypotheses that $\Sigma$ have an
    $H^{\infty}$-transfer function in Proposition \ref{P:WcWo}.

    \smallskip

\paragraph{\bf (2):}
    Take $\cX = \ell^{2}({\mathbb Z}_{+})$, $\cU = {\mathbb C}$ with
 $A$ equal to the forward shift operator and $B$ equal to the
 injection of ${\mathbb C}$ into the first slot of $\ell^{2}({\mathbb
 Z}_{+})$:
$$
A = \sbm{0 & 0 & 0 & \cdots \\ 1 & 0 & 0 & \cdots \\
0 & 1 & 0 & \cdots \\  & & \ddots &  &},
\quad B = \sbm{ 1 \\ 0 \\  \vdots }.
$$
One easily computes that $\bW_c^*$ is the identity operator on
$\ell^{2}({\mathbb Z}_{+})$, hence in particular with dense domain equal to the
whole space. Thus
$\bW_c=I_{\ell^{2}({\mathbb Z}_{+})}$ is bounded. It is also clear
that $\Rea (A|B)=\bW_c \ell_{\tu{fin}}(\BZ_+)=I_{\ell^{2}({\mathbb
Z}_{+})}\ell_{\tu{fin}}(\BZ_+)=\ell_{\tu{fin}}(\BZ_+)\neq \cX$.

\smallskip

\paragraph{\bf (3):}
If $(A,B)$ is exactly controllable, then $\Rea (A | B) = \cX$; thus
trivially $\Rea (A | B)$ is dense in $\cX$, i.e., $(A, B)$ is
controllable.

\smallskip

\paragraph{\bf (4):}
We now assume that $(A,B)$ is exactly $\ell^{2}$-controllable with $\cD(\bW_c^*) = \cX$.
Then  the domain of $\bW_c$ is determined by \eqref{limit-c} where $x$ can be taken to be
an arbitrary vector in $\cX$, i.e., for $\bu \in \ell^2_\cU({\mathbb Z}_-)$,
$$
x_c := \bW_c^* \bu  =  \operatorname{weak-lim}_{K\to
\infty} \sum_{k=-K}^{-1} A^{-k-1} B \bu(k)
$$
Note that each approximant $\sum_{k=-K}^{-1} A^{-k-1} B \bu(k)$ of $x_c$ is in the
reachability space $\Rea (A | B)$.  We conclude that $\im \bW_{c}$ is
contained in the weak-closure of $\Rea (A | B)$. But a consequence
of the Hahn-Banach Theorem is that weak and norm
closure are the same on convex sets (in particular on linear subsets);
hence $\im \bW_{c}$ is contained in the norm-closure of $\Rea (A |
B)$.  The $\ell^{2}$-exact controllability hypothesis now gives us
that $\Rea (A | B)$ is norm-dense in $\cX$, i.e., $(A,B)$ is
controllable.  This verifies item (4).

\smallskip

\paragraph{\bf (5):}
We now assume instead that the pair $(A,B)$ is exactly controllable and
that $\cD(\bW_c^*)$ is dense in $\cX$.  We may then apply Proposition
\ref{P:WcWo'} to see that $\bW_c$ is given by \eqref{limit-c}--\eqref{Wc-act}.
In particular, any $\bu \in \ell^2_{\tu{fin}, \cU}({\mathbb Z}_-)$ is in $\cD(\bW_c)$ with
$\bW_c \bu = \sum_{k \in {\mathbb Z}_-} A^{-k-1} B \bu(k)$ (where the sum is
finite).  This shows that $\im \bW_c \supset \Rea(A|B)$.  The exact controllability
hypothesis now implies that $\im \bW_c = \cX$, i.e., that $(A, B)$ is
$\ell^2$-exactly controllable.
\end{proof}

\section{Infinite-dimensional state-space-similarity theorems}
\label{S:SSS}

The goal of this section is to prove Theorem \ref{T:SSS}.

\begin{proof}[Proof of Theorem \ref{T:SSS}]
The proof is given in four steps. We first prove the sufficiency directions in items (1)-(3), after which we prove the necessity directions in three separate steps.

\paragraph{\em Proof of sufficiency in items {\rm(1)-(3)}}
We first consider the sufficiency direction:  {\em  $\Sigma$ and $\Sigma'$ (pseudo-)similar $\Rightarrow$
$F_\Sigma(\lambda) = F_{\Sigma'}(\lambda)$
in a neighborhood of the origin.} Note that the equality
$F_\Sigma(\lambda) = F_{\Sigma'}(\lambda)$ in a neighborhood of the origin
is the same as matching of Taylor coefficients at the origin:
\begin{equation}  \label{momentseq}
D' = D    \text { and }  C' A^{\prime n} B'  = C A^n B \text{ for } n=0,1,2,\dots.
\end{equation}
Note also that similarity-equivalence between $\Sigma$ and $\Sigma'$ is a particular kind of pseudo-similarity
equivalence.  Hence. to  prove the sufficiency direction in items (1), (2), (3) of Theorem \ref{T:SSS},
it suffices to show: {\em  if $\Gamma$ is a closed, densely defined operator with dense range satisfying conditions
\eqref{pseudosim}, then conditions \eqref{momentseq} hold.}

Toward this end, note first that the condition $D' = D$ is part of the conditions \eqref{pseudosim} (pseudo-similarity
equivalence between $\Sigma$ and $\Sigma'$).  As for the remaining conditions in \eqref{momentseq},  use
the relations in \eqref{pseudosim} to compute
$$
C' A^{\prime n}B' = C' A^{\prime n} \Gamma B = C' \Gamma A^n B = C A^n B
$$
as needed.

\paragraph{\em Proof of necessity in item {\rm(1)}:}
Item (1) has already been worked out in the literature (see \cite{Helton74, BC, Arov79a}), so we only give a sketch.
We suppose that we are given two minimal systems $\Sigma$ and $\Sigma'$ with respective system matrices
$M = \sbm{ A & B \\ C & D}$ and $M' = \sbm{ A' & B' \\ C' & D' }$ with conditions \eqref{momentseq} holding.
We must construct an injective, closed, densely defined operator $\Gamma$ with dense range so that \eqref{pseudosim} holds.
Toward this end, we attempt to define an operator $\Gamma_0$ from $\Rea(A|B)$ to $\Rea(A'|B')$ by
\begin{equation}  \label{defGamma0}
  \Gamma_0 \colon \sum_{k=0}^n A^k B u_k  \mapsto \sum_{k=0}^n A^{\prime k} B' u_k.
\end{equation}
One can use the observability of the pair $(C', A')$ to see that the formula for $\Gamma_0$ is well-defined
and observability of the pair $(C,A)$ to see that the resulting well-defined linear transformation
$\Gamma_0$ is injective.  Furthermore,
controllability of the pair $(A,B)$ implies that $\Gamma_0$ has dense domain and controllability
of the pair $(A',B')$ implies that $\Gamma_0$ has dense range $\im \Gamma_0$ in $\cX'$.
A mild limit enhancement of these computations shows that moreover $\Gamma_0$ is closable with closure $\Gamma$
also injective with dense range.

From the definition \eqref{defGamma0} of the action of $\Gamma_0$, it is clear that $\Gamma B = B'$ and that
$\Gamma A x = A' \Gamma x$ if $x \in \Rea(A|B)$.  A limit enhancement of this same argument then shows that
$\Gamma A x = A' \Gamma x$ for any $x \in \cD(\Gamma)$.  Similarly, application of the operator $C$ to an element
$x = \sum_{k=0}^n A^k B u_k$ combined with the equality of Taylor coefficients \eqref{momentseq} and the definition
\eqref{defGamma0} of the action of $\Gamma_0$ yields the identity $C x = C' \Gamma x$ for $x \in \Rea(A|B)$.
A limit enhancement of this argument then gives the equality $C x = C' \Gamma x$ for a general $x$ in
$\cD(\Gamma)$.  We conclude that $\Gamma$ implements a pseudo-similarity equivalence between $\Sigma$ and $\Sigma'$ as wanted.

\paragraph{\em Proof of necessity in item {\rm(2)}:}  In this case we are given that $\Sigma$ is exactly minimal while $\Sigma'$
 is minimal such that relations  \eqref{momentseq} hold. Trivially, $\Sigma$ then is also minimal.
 The work in the immediately preceding proof (necessity in item (1))
 then tells us that the operator $\Gamma_0$ defined on $\Rea(A|B)$ by \eqref{defGamma0} is well-defined and injective with dense range, and moreover is closable.  The exact minimality hypothesis on $\Sigma$ means in particular that
 $(A,B)$ is exactly controllable, i.e., that the reachability space $\Rea(A|B)$ is the whole space $\cX$.
 Hence, the closability of $\Gamma_0$ just means that $\Gamma_0$ is a closed operator with domain equal to
 the whole space $\cX$.  The Closed Graph Theorem then implies that $\Gamma_0$ is bounded as an operator from
 $\cX$ to $\cX'$.  Moreover, by the work in the proof for item (1) above, we know that $\Gamma_0$ satisfies all the
 relations in \eqref{pseudosim}.  It remains only to show that $\Gamma_0$ is surjective.  It then follows that $\Gamma$
 has a bounded inverse by the Open Mapping Theorem.

 Toward this end, we view $M^* = \sbm{ A^* & C^* \\ B^* & D^* }$ as the system matrix for a linear system $\Sigma^*$
 and  similarly $M^{\prime *} = \sbm{ A^{\prime *} & C^{\prime *} \\ B^{\prime *} & D^*}  $ as a system matrix for a linear system $\Sigma^{\prime *}$.  Note that
$$
F_{\Sigma^*}(\lambda) = F_{\Sigma}(\overline{\lambda})^* = F_{\Sigma'}(\overline{\lambda})^* =
F_{\Sigma^{\prime *}}(\lambda)
$$
so $\Sigma^*$ and $\Sigma^{\prime *}$ have identical transfer functions in a neighborhood of the origin, and hence
$$ D^{\prime *} = D^* \text{ and } B^{\prime *} A^{\prime *n} C^{\prime *} = B^* A^{*n} C^* \text{ for } n=0,1,2,\dots
$$
(just the adjoint versions of the relations   \eqref{momentseq}).  Moreover, the $\ell^2$-exact observability of $\Sigma$
implies that $\Sigma^*$ is $\ell^2$-exactly controllable and the observability of $\Sigma'$ implies that
$\Sigma^{\prime *}$ is controllable.  We may then repeat the preceding argument but applied to the pair $(\Sigma^*,
\Sigma^{\prime *})$ in place of the pair $(\Sigma, \Sigma')$.  We conclude that there is a well-defined bounded linear operator $\widetilde \Gamma$ from $\cX$ to $\cX'$ uniquely determined by its action on vectors $x$ of the form
$x = \sum_{k=0}^n A^{*k} C^* y_k$:
$$
\widetilde \Gamma \colon \sum_{k=0}^n A^{*k} C^* y_k \mapsto \sum_{k=0}^n A^{\prime * k} C^{\prime *} y_k,
$$
which in addition satisfies the intertwining relations:
$$
\widetilde \Gamma A^* = A^{\prime *} \widetilde \Gamma, \quad
\widetilde \Gamma C^* = C^{\prime *}, \quad B^{\prime *} \widetilde \Gamma = B^*.
$$
In other words, $\widetilde \Gamma^*$ satisfies
$$
A \widetilde \Gamma^* = \widetilde \Gamma^* A', \quad C \widetilde \Gamma^* = C',
\quad \widetilde \Gamma^* B' = B.
$$
A consequence of these relations is that
$$
  \widetilde \Gamma^* \colon \sum_{k=0}^n A^{\prime k} B' u_k = \sum_{k=0}^n A^k B u_k
$$
for any choice of $u_k \in \cU$, $k = 0,1, \dots, n$.  This implies that $\widetilde \Gamma^* \Gamma x = x$
for all $x \in \Rea(A|B) = \cX$.  Thus $\widetilde \Gamma^*$ is a bounded left inverse of $\Gamma$.

We use this last observation to see that $\im \Ga$ is closed as follows.  If
$x'_{n} = \Gamma x_{n}$ is a sequence of elements of $\im
\Ga$ converging to $ x' \in  \cX'$, then
$x_{n} = \widetilde \Ga^{*} \Ga x_{n} \to \widetilde \Ga^{*}
x' \in \cX$ as $n \to \infty$.  Since $\Ga$ is bounded, we
conclude that
$$
x' = \lim_{n \to \infty} \Gamma x_{n}
= \Gamma (\lim_{n \to \infty} x_{n}) = \Gamma \widetilde \Gamma^{*}
x'
\in \im \Gamma
$$
and we conclude that $\im \Gamma$ is closed as claimed. As $\im
\Gamma$ is also dense due the assumed controllability of the pair
$(A', B')$, it follows that $\im \Gamma$ is the whole space $\cX'$
and in fact that $\widetilde \Gamma^{*}$ is a two-sided bounded
inverse for
$\Gamma$, as needed to complete the proof.

\paragraph{\em Proof of necessity in item {\rm(3)}:}  We are now given that $\Sigma$ is exactly $\ell^2$-minimal while $\Sigma'$
is assumed to be  minimal with bounded controllability and
observability operators $\bW'_c$ and $\bW'_o$  and furthermore the relations \eqref{momentseq} hold.
We must produce a bounded, boundedly invertible operator $\Gamma \colon \cX \to \cX'$ so that the relations
\eqref{pseudosim} hold.

By Corollary \ref{C:ell2implics}, the operators $\bW_o$ and $\bW_c^*$
are bounded operators which are also bounded below. In particular, $\bW_o$  admits a bounded
generalized left-inverse $\bW_o^\dagger$ and $\bW_c$ admits a bounded generalized right-inverse $\bW_c^\dagger$, i.e.,
\begin{equation}  \label{gen-inv}
\bW_o^\dagger\bW_o=I_\cX, \qquad \bW_c\bW_c^\dagger=I_\cX.
\end{equation}
In addition we may choose $\bW_o^\dagger$ and $\bW_c^\dagger$ to be the Moore-Penrose generalized inverses;
this means that in addition to \eqref{gen-inv} we have
\begin{equation}  \label{MoorePenrose}
\bW_o \bW_o^\dagger = P_{ \im \bW_o}, \quad
\bW_c^\dagger \bW_c = P_{(\ker \bW_c)^\perp}
\end{equation}
where in general $P_\cN$ indicates the orthogonal projection onto the subspace $\cN$.
Furthermore, from the fact that $\bW_c$ and $\bW_o$ are bounded, we see that $\fH_{F_\Sigma} = \bW_o \bW_c$ is
bounded and similarly $\fH_{F_{\Sigma'}} = \bW'_o \bW'_c$.  From the assumption that $F_\Sigma = F_{\Sigma'}$ in a
neighborhood of the origin, it follows that $\fH_{F_\Sigma} = \fH_{F_{\Sigma'}}$ and hence
\begin{equation}   \label{Hankel-id}
  \bW_o \bW_c = \bW'_o \bW'_c.
\end{equation}
A consequence of this property combined with the observability of the output pair $(C', A')$ (i.e., the
injectivity of the operator $\bW'_o$) is the fact that
\begin{equation}   \label{key}
   \bW'_c|_{\ker \bW_c} = 0.
\end{equation}

Let us define $\Gamma \colon \cX \to \cX'$ by
\begin{equation}   \label{defGamma}
   \Gamma = \bW'_c \bW_c^\dagger.
\end{equation}
As both $\bW'_c$ and $\bW_c^\dagger$ are bounded, we see that $\Gamma$ is a bounded operator.

We next use \eqref{MoorePenrose} and \eqref{key} to check that
\begin{equation}  \label{Gamma-prop}
  \Gamma \bW_c = \bW'_c
\end{equation}
as follows:
$$
 \Gamma \bW_c = \bW'_c \bW_c^\dagger \bW_c = \bW'_c P_{(\ker \bW_c)^\perp} = \bW'_c.
$$

Set $\widetilde \Gamma = \bW_o^\dagger  \bW'_o$.
Let us check that $\widetilde \Gamma$ is a left inverse for $\Gamma$:
\begin{align*}
  \widetilde \Gamma \Gamma & = \bW_o^\dagger \bW'_o \bW'_c \bW_c^\dagger  \\
 & =  \bW_o^\dagger \bW_o \bW_c \bW_c^\dagger \quad \text{ by \eqref{Hankel-id}} \\
 & = I_\cX\quad \text{ (by \eqref{gen-inv}).}
  \end{align*}
Note next that $\Gamma$ has dense range by the assumed controllability of the system $\Sigma'$ and the relation
 \eqref{Gamma-prop}.

To show that $\widetilde \Gamma$  is a two-sided inverse for $\Gamma$, it suffices to show that
$\im \Gamma$ is closed.  As $\widetilde \Gamma$ is a bounded left inverse for $\Gamma$, this follows by exactly
the same argument as used at the end of the proof of the sufficiently in item (2) given immediately above.

To verify that $\Gamma$ implements a similarity equivalence between $\Sigma$ and $\Sigma'$, it now
remains only to verify the intertwining conditions \eqref{pseudosim}.   Toward this end, let us point out that it is easily
verified from the definitions that  the following intertwining condition holds:
$$
\bW_c^* A^*  = \cS_- \bW_c^*
$$
where $\cS_-$ is the truncated right shift operator on $\ell^2_\cU({\mathbb Z}_-)$.
Taking adjoints then gives us
$$
  \bW_c  \cS_-^* = A \bW_c
 $$
 where $\cS_-^*$ is the (untruncated) backward shift operator on $\ell^2_\cU({\mathbb Z}_-)$.
 Making use of \eqref{Gamma-prop} we then get
 $$
 \Gamma A \bW_c = \Gamma \bW_c \cS_-^* = \bW'_c \cS_-^* = A' \bW'_c = A' \Gamma \bW_c
 $$
 and we arrive at  the first intertwining condition in \eqref{pseudosim}:
 $$
  \Gamma A  = A' \Gamma.
 $$

To verify the second intertwining condition ($\Gamma B = B'$) in \eqref{pseudosim}, observe that $Bu
= \bW_c \bu$ where $\bu(-1) = u$ an $\bu(k) = 0$ for $k < -1$.  Hence
$$
  \Gamma B u = \Gamma \bW_c \bu = \bW'_c \bu = B' u
 $$
 as wanted.  To see the last intertwining condition ($C' \Gamma = C$), simply note first that,
 for any $\bu \in \ell^2_{{\rm fin}, \cU}({\mathbb Z}_-)$,  as a consequence of the identities
 \eqref{momentseq} we have
 $$
  C' \bW'_c  \bu = C' \left(\sum_{k=-1}^{-K} A^{\prime -k-1} B' u(k)\right) =
  C \left(\sum_{k=1}^{-k} A^{-k-1} B' u(k) \right) = C \bW_c \bu.
  $$
  By approximating an arbitrary $\bu \in \ell^2_\cU({\mathbb Z}_-)$ by input signals of finite support and taking
  limits,  we arrive at the general operator identity
  $$
  C' \bW'_c = C \bW_c.
  $$
  Hence, by combining this with the identity \eqref{Gamma-prop} we can compute
 $$
  C' \Gamma \bW_c = C' \bW'_c = C \bW_c
  $$
  and arrive at the last of the intertwining relations \eqref{pseudosim} as wanted.
  The completes the proof of necessity in item (3) of Theorem \ref{T:SSS}.
  \end{proof}

  \begin{remark}  \label{R:sim-vs-pseudosim} {\em Similarity versus Pseudo-similarity.}
  The result of the sufficiency side in Theorem \ref{T:SSS} is that the existence of a similarity or even only
  pseudo-similarity transform from $\Sigma$ to $\Sigma'$  is enough to ensure that $F_\Sigma(\lambda)
  = F_{\Sigma'}(\lambda)$ for $\lambda$ in a neighborhood of the origin.    It can easily be checked that the existence
  of a similarity transform from $\Sigma$ to $\Sigma'$ (in the sense used in Theorem \ref{T:SSS}) preserves most other
  system-theoretic properties which we have discussed so far, namely:  exponential stability;  controllability,
  exact controllability, $\ell^2$-exact controllability; and hence also by duality observability, exact observability,
  $\ell^2$-exact observability; and therefore also minimality, exact minimality, and $\ell^2$-exact minimality.
  On the other hand, identifying which properties are preserved under pseudo-similarity equivalence is much more
  delicate.  For example, it is possible to produce an exponentially stable state operator $A$ which is pseudo-similar
  to a state operator $A'$ which is not exponentially stable (see \cite[Section 2.7]{AKP06}).
  If $\Gamma$ is a pseudo-similarity from $\Sigma \sim (A,B,C,D)$ to $\Sigma' \sim (A', B', C', D')$ and
  $(A,B)$ is controllable,  then one can show that $(A',B')$ is again controllable if one
  imposes the additional hypothesis that
  $\Rea(A|B) \subset \cD(\Gamma)$ is a {\em core} for $\Gamma$, i.e., given $x \in \cD(\Gamma)$, there
  exists a sequence $\{ x_n \}$ contained in $\Rea(A|B)$ so that $x_n \to x$ and $\Gamma x_n \to \Gamma x$ as
  $n \to \infty$.   This same hypothesis that $\Rea(A|B)$ be a core for $\Gamma$ (or equivalently for $H^\half =
  (\Gamma^* \Gamma)^\half$) comes up in \cite{AKP06} in the discussion of characterization of maximal and minimal
  solutions of the KYP-inequality.
   \end{remark}

  \section{Infinite-dimensional standard bounded real lemmas}
\label{S:infstanBRL}

In this section we prove Theorem \ref{T:BRLinfstan}. The following lemma connects (generalized) solutions to the KYP inequality to (pseudo) similarity. Note that no minimality condition is assumed.

\begin{lemma}\label{L:SSSvsKYP}
Let $\Si$ be a discrete-time linear system as in \eqref{dtsystem} with system matrix $M$ as in \eqref{sysmat} and transfer function $F_\Si$ defined by \eqref{trans}. Then:
\begin{itemize}
  \item[(1)] $\Si$ is similar to a contractive system if and only if there exists a bounded, strictly positive-definite solution to the KYP inequality \eqref{KYP1}.

  \item[(2)] $\Si$ is pseudo-similar to a contractive system if and only if there exists a generalized positive-definite solution to the spatial KYP inequality \eqref{KYP1b}.

\end{itemize}
\end{lemma}

\begin{proof}
We begin with a proof of item (1).
Let $H$ be a bounded strictly pos\-itive-definite solution of the KYP-inequality \eqref{KYP1}. This is equivalent to the system matrix of the discrete-time linear system $\Sigma'$ associated with the quadruple $\{H^\half A H^{-\half}, H^{\half}B, C H^{-\half},D\}$ being contractive. Hence $\Si$ is similar to a contractive system.

Conversely, assume $\Si$ is similar to a contractive system $\Si'=\{A',B',C',D'\}$ via a bounded and boundedly invertible operator $\Ga:\cX\to\cX'$, i.e., $A'$, $B'$, $C'$ and $D'$ are given by \eqref{sys-sim}. Set $H=\Ga^*\Ga$ and let $M'$ denote the system matrix of $\Si'$. Then $M'$ being contractive implies
\begin{align*}
0 & \preceq
I-M'^* M'
=\mat{cc}{I&0\\0&I}-\mat{cc}{ \Ga^{-1*}A^*\Ga^* & \Ga^{-1*}C^* \\ B\Ga^* & D^* }
\mat{cc}{\Ga A\Ga^{-1} & \Ga B \\ C\Ga^{-1} & D}\\
 & =\mat{cc}{\!\!\Ga^{-1*}&0 \!\!\\ \!\! 0&I \!\!}\left( \mat{cc}{H&0\\0&I}-\mat{cc}{A^*&B^*\\C^*&D^*}\mat{cc}{H&0\\0&I}\mat{cc}{A&B\\ C&D}
 \right)\mat{cc}{\!\!\Ga^{-1}&0\!\!\\ \!\! 0&I\!\!}.
\end{align*}
Thus the KYP-inequality \eqref{KYP1} holds with $H=\Ga^*\Ga$.

Next we prove item (2). The idea behind the proof is the same as for item (1), but one has to be more careful
when dealing with generalized KYP solutions and pseudo-similarity. First assume there exists a generalized
positive-definite solution $H$ to the spatial KYP inequality, i.e., $H$ is closed, densely defined, injective
positive-definite operator on $\cX$ satisfying \eqref{KYP1b'} and \eqref{KYP1b}.
We now define operators
\[
A':\im H^\half \to \cX,\quad B':\cU\to \cX,\quad
C':\im H^\half\to \cY,\quad D'=D:\cU\to\cY
\]
via
\begin{equation} \label{A'B'C'D'}
\mat{cc}{A' \!&\! B'\\ C' \!&\! D'}\mat{c}{\! H^\half x \! \\\! u\!}
=\mat{cc}{H^\half  \!&\!  0\\ 0 \!&\! I}\mat{cc}{A \!&\! B\\ C \!&\! D}
\mat{c}{x\\u}
\  (x\in\cD(H^\half),\, u\in\cU).
\end{equation}
Note that the right-hand side is well defined for each $x\in\cD(H^\half)$ and $u\in\cU$ because of \eqref{KYP1b'}.
Since $H^\half$ is injective, $H^\half x\oplus u$ in $\im H^\half \oplus\cU$ being equal to the zero vector implies
$x=0$ and $u=0$. This implies that the operators $A'$, $B'$, $C'$ and $D'$ are well defined on their given domains.
Moreover, as a consequence of the Spectral Theorem for unbounded selfadjoint operators (see e.g.\
\cite[Theorem VIII.6]{RS}), one can see that $H$ being selfadjoint and injective implies that  $H$ as well as
$H^\half$ have dense range in $\cX$.  Thus the block operator matrix $\sbm{ A' & B' \\ C' & D'}$  given by
\eqref{A'B'C'D'} has dense domain in  $\cD(H^\half) \oplus \cU$ in $\cX \oplus \cU$.  Furthermore,
by the spatial KYP-inequality \eqref{KYP1b},  we see that this operator \eqref{A'B'C'D'}
acts contractively on  its domain, and hence can be continuously extended to a $2\times 2$ block operator
(of which the entries are also denoted by $A'$, $B'$, $C'$, $D'$) that maps $\cX\oplus \cU$ contractively into
$\cX\oplus \cY$. Hence the operators $\{A',B',C',D'\}$ generate a contractive linear system $\Si'$. We claim
that $\Ga=H^\half$ provides a pseudo-similarity between the systems $\Si$ and $\Si'$. By definition $H^\half$ is
densely defined, and we already observed above that $H^\half$ has dense range. The remaining conditions on
$\Ga=H^\half$ listed in \eqref{pseudosim} follow directly from \eqref{KYP1b'} and the definition of the operators
$A'$, $B'$, $C'$ and $D'$.

It remains to prove the reverse inclusion. Hence, we assume $\Si$ is pseudo-similar to a contractive system $\Si'$
given by the quadruple $\{A',B',C',D'\}$ via the pseudo-similarity $\Ga:\cX\to \cX'$.  Since $\Ga$ is a closed operator,
by \cite[Theorem VIII.32]{RS} it admits a polar decomposition $\Ga=U |\Ga|$ with $|\Ga|$ the positive-semidefinite,
square root $|\Ga|=(\Ga^*\Ga)^\half$ with $\cD(|\Ga|)=\cD(\Ga)$ and $U$ a partial isometry with initial space equal to
 $(\kr \Ga)^\perp=\cX$ and final space $\overline{\im \Ga}=\cX'$, i.e., $U$ is unitary.  As $\Gamma$ is injective,
 in fact $|\Ga|$ is positive-definite.
 Now set $H=\Ga^*\Ga=|\Ga|^2$, so that $H^\half=|\Ga|$.   From the Spectral Theorem applied to  $|\Gamma| = H^\half$,
 one can read off that $H = (H^\half)^2$ is positive-definite selfadjoint with a domain in general smaller than
 $\cD(H^\half)$ but still dense in $\cX$.  Since $\cD(H^\half)=\cD(\Ga)$, the inclusions \eqref{KYP1b'} follow directly
 from \eqref{pseudosim}.  Since $U$ is unitary, for each $x\in\cD(H^\half)=\cD(\Ga)$  and $u\in\cU$ we have
\begin{align*}
& \left\|\mat{cc}{H^\half&0\\0&I}\mat{c}{x\\u}\right\|^2 - \left\|\mat{cc}{H^\half&0\\0&I}
\mat{cc}{A&B\\C&D}\mat{c}{x\\u}\right\|^2=\\
&\quad =  \left\|\mat{cc}{UH^\half&0\\0&I}\mat{c}{x\\u}\right\|^2 - \left\|\mat{cc}{UH^\half&0\\0&I}
\mat{cc}{A&B\\C&D}\mat{c}{x\\u}\right\|^2\\
&\quad =  \left\|\mat{cc}{\Ga&0\\0&I}\mat{c}{x\\u}\right\|^2 - \left\|\mat{cc}{\Ga&0\\0&I}
\mat{cc}{A&B\\C&D}\mat{c}{x\\u}\right\|^2\\
&\quad =  \left\|\mat{c}{\Ga x\\u}\right\|^2 \!\!-\! \left\|\mat{cc}{A'\Ga&B'\\C'\Ga&D'}\mat{c}{x\\u}\right\|^2
\!\!=\!  \left\|\mat{c}{\Ga x\\u}\right\|^2 \!\!-\! \left\|\mat{cc}{A'&B'\\C'&D'}\mat{c}{\Ga x\\u}\right\|^2\!\!.
\end{align*}
The fact that $\Si'$ is a contractive system then shows that \eqref{KYP1b} holds. Hence $H$ is a generalized
positive-definite solution to the spatial KYP inequality associated with $\Si$.
\end{proof}

\begin{proof}[Proof of Theorem \ref{T:BRLinfstan}]
We start with the sufficiency claims of items (2) and (3). In both cases, we assume that the KYP inequality \eqref{KYP1}
has a bounded, strictly positive-definite solution $H$. Then item (1) of Lemma \ref{L:SSSvsKYP} yields that $\Si$ is
similar to a contractive system. The sufficiency claims of items (2) and (3) then follow directly from the sufficiency
claims of items (2) and (3) of Theorem \ref{T:SSS}. Moreover, since $\Si$ is similar to a contractive system, $A$ is similar
 to the state operator of a contractive system, which is a contraction. In particular, the spectral radius of $A$ is at
 most one, so that the transfer function $F_\Si$ is analytic on $\BD$. The sufficiency in item (1) follows in the
 same way, now combining item (2) of Lemma \ref{L:SSSvsKYP} with the sufficiency direction of item (1) of
 Theorem \ref{T:SSS}.

Next we proof the necessity claims of Theorem \ref{T:BRLinfstan}. In all three items we assume the the transfer function $F_\Si$ of $\Si$ has an analytic continuation to a Schur class function in $\cS(\cU,\cY)$. By the contractive realization theorem for Schur class functions
(see e.g.\ \cite[Theorem 5.2]{BC} or \cite[Theorem VI.3.1]{NF} as
well as \cite{AKP97}), there is another discrete-time linear system
$\Si'$ with a contractive system matrix
$$
M' = \begin{bmatrix} A' & B' \\ C' & D  \end{bmatrix}\colon
\begin{bmatrix} \cX' \\ \cU \end{bmatrix} \to \begin{bmatrix} \cX' \\
    \cY \end{bmatrix}
$$
for some Hilbert space $\cX'$ with associated transfer function
$F_{\Si'}$  equal to $F_\Si$ on an open neighborhood of $0$. By
compressing orthogonally to the controllable and observable subspace
in $\cX'$, we may assume that the system $\Si'$ is minimal, i.e., the
input pair $(A',B')$ is controllable and the output pair $(C',A')$ is
observable. The fact that $M'$ is contractive implies that the controllability operator $\bW_c'$ and the
observability operator $\bW_o'$ associated with $\Si'$ are both contractive, by Proposition~\ref{P:contractiveSM}.

It now follows that for $\textup{d}=1,2,3$, assuming the conditions of item (d) in Theorem \ref{T:BRLinfstan},
all hypotheses of item (d) in Theorem \ref{T:SSS} are satisfied, so that we can conclude that $\Si$ and $\Si'$
are pseudo-similar for $\textup{d}=1$  and similar for $\textup{d}=2,3$. The claims then follow from
Lemma \ref{L:SSSvsKYP}, applying item (1) for $\textup{d}=2,3$ and item (2) for $\textup{d}=1$.
\end{proof}

\section{The infinite-dimensional strict bounded real lemma}
\label{S:infstrictBRL}

In this section we prove Theorem \ref{T:BRLinfstrict}.

\begin{proof}[Proof of sufficiency in Theorem
    \ref{T:BRLinfstrict}]
The sufficiency follows simply from the sufficiency in Theorem
\ref{T:BRLinfstan} (2). Indeed, since the KYP-inequality in question
\eqref{KYP2} is strict, one can replace $B$ and $D$ by $\ga B$ and
$\ga D$ for a sufficiently small $\ga>1$ without violating the strict
inequality. Evoking the sufficiency claim of Theorem
\ref{T:BRLinfstan} (2) tells us that $\ga F_\Si$ is a Schur class
function, so that $\| F_\Si\|_\infty\leq 1/\ga <1$. Hence
$F_\Si\in\cS^{o}(\cU, \cY)$.
\end{proof}

Our proof of the necessity also relies on Theorem \ref{T:BRLinfstan},
but is more involved. We follow the ideas from the proof for the
finite-dimensional case from Petersen-Anderson-Jonkheere \cite{PAJ}.

\begin{proof}[Proof of  necessity in Theorem
\ref{T:BRLinfstrict}]

Let $\Si$ be a discrete-time linear system as in \eqref{dtsystem}
with system matrix $M$ as in \eqref{sysmat} and transfer function
$F_\Si$ defined by \eqref{trans}. Assume that (i) $\spec (A) <1$, and
(ii) $F_{\Sigma}$ is in the strict Schur class $\cS^{o}(\cU, \cY)$.

Since $\spec (A) < 1$, we have that the resolvent expression $(I - zA)^{-1}$ is
uniformly bounded in norm with respect to  $z$ in the unit disk ${\mathbb D}$.  It
follows that we can choose $\epsilon >0$ sufficiently small so that the
augmented matrix function
\begin{equation}   \label{Fepsilon}
F_{\epsilon}(z) : = \begin{bmatrix} F(z) & \epsilon z C (I -
    zA)^{-1} \\ \epsilon z (I - zA)^{-1} B & \epsilon^{2} z (I -
    zA)^{-1} \\ \epsilon I_{\cU} & 0 \end{bmatrix}
\end{equation}
is in the strict Schur class $\cS^{o}( \cU \oplus \cX, \cY \oplus \cX
\oplus \cU)$. Note that
\[
 F_{\epsilon}(z) = \begin{bmatrix} D & 0 \\ 0 & 0 \\ \epsilon I_{\cU}
 & 0 \end{bmatrix} + z \begin{bmatrix}  C \\ \epsilon I_{\cX} \\ 0
 \end{bmatrix} (I - zA)^{-1} \begin{bmatrix} B & \epsilon I_{\cX}
 \end{bmatrix}
\]
and hence
\begin{equation}   \label{breal}
 M_{\epsilon} =  \begin{bmatrix}  \bA & \bB \\ \bC & \bD
\end{bmatrix} : =
  \mat{c|cc}{
    A &  B & \epsilon I_{\cX}\\
    \hline C &  D & 0 \\
    \epsilon I_{\cX} & 0 & 0 \\
    0  & \epsilon I_{\cU} & 0}
\end{equation}
is a realization for $ F_{\epsilon}(z)$ with associated linear system which we denote by $\Sigma_\epsilon$.
Note that $\bB$ is already
onto the state space $\cX$ and $\bC^{*}$ is also onto $\cX$, so the
system $\Sigma_\epsilon$  is exactly controllable and exactly
observable, i.e., exactly minimal.  As $A$ is exponentially stable, it is also the case that $\Sigma_\epsilon$ is
$\ell^2$-exactly minimal.  We may therefore apply either of items (2) or (3) in Theorem \ref{T:BRLinfstan}
to conclude that there is a bounded strictly positive-definite operator $H$ on the state space $\cX$ so
that
$$
   \begin{bmatrix} \bA^{*} & \bC^{*} \\ \bB^{*} & \bD^{*}
\end{bmatrix}
       \begin{bmatrix} H & 0 \\ 0 & I_{\cY \oplus \cX \oplus \cU}
       \end{bmatrix} \begin{bmatrix} \bA & \bB \\ \bC & \bD
   \end{bmatrix} \preceq \begin{bmatrix} H & 0 \\ 0 & I_{\cU \oplus
   \cX} \end{bmatrix}.
$$
Spelling this out gives
$$
\begin{bmatrix} A^{*}HA + C^{*}C + \epsilon^{2}I_{\cX} & A^{*}H B +
    C^{*}D & \epsilon A^{*}H \\
 B^{*}HA + D^{*}C & B^{*} H B + D^{*} D + \epsilon^{2}I_{\cU} &
 \epsilon B^{*} H \\ \epsilon HA & \epsilon HB & \epsilon^{2} H
\end{bmatrix} \preceq \begin{bmatrix} H & 0 & 0 \\ 0 & I_{\cU} & 0 \\
0 & 0 & I_{\cX} \end{bmatrix}.
$$
By crossing off the third row and third column, we get the inequality
$$
\begin{bmatrix} A^{*} H A + C^{*} C + \epsilon^{2} I_{\cX} & A^{*}H B
    + C^{*} D \\ B^{*}H A + D^{*} C & B^{*} H B + D^{*} D +
    \epsilon^{2} I_{\cU} \end{bmatrix} \preceq \begin{bmatrix} H & 0
    \\ 0 & I_{\cU} \end{bmatrix}
$$
or
$$
\begin{bmatrix} A^{*} & C^{*} \\ B^{*} & D^{*} \end{bmatrix}
    \begin{bmatrix} H & 0 \\ 0 & I_{\cY} \end{bmatrix}
	\begin{bmatrix} A & B \\ C & D \end{bmatrix} + \epsilon^{2}
	    \begin{bmatrix} I_{\cX} & 0 \\ 0 & I_{\cU} \end{bmatrix}
		\preceq \begin{bmatrix} H & 0 \\ 0 & I_{\cU}
	    \end{bmatrix}
$$
leading us to the strict KYP-inequality \eqref{KYP2} as wanted.
 \end{proof}

\paragraph{\bf acknowledgements}
This work is based on the research supported in part by the National
Research Foundation of South Africa (Grant Numbers 93039, 90670, and
93406). Any opinion, finding and conclusion or
recommendation expressed in this material is that of the authors and
the NRF does not accept any liability in this regard.

It is a pleasure to thank Mikael Kurula for his useful comments leading to a number of
improvements in the exposition while visiting the first author in July 2017.

\end{document}